\documentclass[reqno]{amsart}

\makeatletter
\@namedef{subjclassname@2020}{%
  \textup{2020} Mathematics Subject Classification}
\def\subsection{\@startsection{subsection}{2}%
  \z@{.5\linespacing\@plus.7\linespacing}{.25\linespacing}%
  {\normalfont\bfseries}}
\def\subsubsection{\@startsection{subsubsection}{3}%
  \z@{.5\linespacing\@plus.7\linespacing}{.25\linespacing}%
  {\normalfont\itshape}}
\makeatother

\usepackage{amsmath,amssymb,mathtools,url,xcolor}
\usepackage[hidelinks]{hyperref}

\numberwithin{equation}{section}

\newtheorem{theorem}{Theorem}[section]
\newtheorem{lemma}[theorem]{Lemma}
\newtheorem{proposition}[theorem]{Proposition}

\theoremstyle{definition}
\newtheorem{definition}[theorem]{Definition}
\newtheorem{assumption}{Assumption}

\theoremstyle{remark}
\newtheorem{remark}[theorem]{Remark}

\newcommand{\dd}{\,\mathrm d}
\newcommand{\cU}{\mathcal U}
\newcommand{\cG}{\mathsf G}
\newcommand{\cI}{\mathcal I}
\newcommand{\cM}{\mathsf M}
\newcommand{\cP}{\mathsf P}
\newcommand{\cQ}{\mathsf Q}
\newcommand{\cZ}{\mathcal Z}
\newcommand{\Ent}{\mathrm H}
\newcommand{\E}{\mathbb E}
\newcommand{\TV}{\mathrm{TV}}
\newcommand{\esssup}{\operatorname*{ess\,sup}}
\newcommand{\Id}{\mathrm{Id}}
\newcommand{\dettwo}{\operatorname{det}_2}
\newcommand{\abs}[1]{\lvert #1\rvert}
\newcommand{\norm}[1]{\lVert #1\rVert}

\begin{document}

\title[Uniform Partition-Function Estimates]
{Uniform Partition-Function Estimates for Coulomb Modulated Energy
at All Positive Temperatures}

\author{Zhenfu Wang}
\address{Beijing International Center for Mathematical Research,
Peking University, 5 Yiheyuan Road, \mbox{Beijing 100871}, China}
\email{zwang@bicmr.pku.edu.cn}

\author{Xianliang Zhao}
\address{Beijing International Center for Mathematical Research,
Peking University, 5 Yiheyuan Road, \mbox{Beijing 100871}, China}
\email{xzhaomath@gmail.com}

\hypersetup{
  pdftitle={Uniform Partition-Function Estimates for Coulomb Modulated Energy at All Positive Temperatures},
  pdfauthor={Zhenfu Wang and Xianliang Zhao},
  pdfsubject={Uniform centered partition-function estimates for Coulomb--Riesz modulated energy at all positive temperatures},
  pdfkeywords={partition function, propagation of chaos, Coulomb interaction, relative entropy, Gaussian fluctuations}
}

\begin{abstract}
We prove uniform-in-\(N\) partition-function estimates at all positive
temperatures for the centered modulated energy of logarithmic, Riesz, and
Bessel--Riesz kernels on \(\mathbb R^d\) in the locally square-integrable range
\(0\le s<d/2\).  They yield finite-particle exponential integrability,
quadratic behavior at small parameter, explicit kernel-dependent growth at
large parameter, and uniform entropy control of the associated Gibbs
measures.  As applications, we obtain uniform R\'enyi-divergence and
relative-entropy bounds for interacting Gibbs equilibria and entropic
mean-field convergence near thermal equilibrium; the optimal \(N^{-1}\)
normalized relative-entropy rate for the three-dimensional Coulomb flow; and
an \(N^{-1/2}\) Gaussian approximation for fixed-time finite-dimensional
fluctuations.  Finally, the modulated energy converges to a random variable in
the second Wiener chaos, and for every positive parameter the partition
functions converge to its Laplace transform, represented by a
Carleman--Fredholm determinant.  This formula identifies the sharp
small-parameter behavior and, when the reference density is bounded below on
a ball, the sharp large-parameter growth.
\end{abstract}



\maketitle

\begingroup
\setcounter{tocdepth}{1}
\tableofcontents
\endgroup

\section{Introduction}

Mean-field models replace a large interacting particle system by a deterministic
continuum density.  Beyond this law of large numbers scale, the centered
interaction energy governs fluctuations of the particles, next-order Gibbs
weights, and quantitative stability of the mean-field approximation.  This
article studies the exponential integrability of that energy for singular
Coulomb--Riesz interactions.  Our central goal is an estimate that is uniform
in the particle number and valid at every positive temperature, including the
nonperturbative low-temperature regime.

To formulate the problem, fix a probability density \(\rho\), let
\(X_1,\ldots,X_N\) be independent with law \(\rho(x)\dd x\), and write
\(X^{1:N}=(X_1,\ldots,X_N)\).  For an even interaction potential \(g\), define
its centered version relative to \(\rho\) by
\[
 g_\rho^\circ(x,y)
 :=g(x-y)-(g*\rho)(x)-(g*\rho)(y)
   +\iint_{\mathbb R^d\times\mathbb R^d}
      g(u-v)\rho(u)\rho(v)\dd u\dd v.
\]
The kernel \(g_\rho^\circ\) has zero \(\rho\)-mean in each variable:
\begin{equation}\label{eq:canonical-centering}
 \int_{\mathbb R^d}g_\rho^\circ(x,v)\rho(v)\dd v=0,
 \qquad
 \int_{\mathbb R^d}g_\rho^\circ(u,y)\rho(u)\dd u=0,
 \qquad  \forall \, x,y\in\mathbb R^d.
\end{equation}
Following Serfaty's terminology \cite{Serfaty2020}, we define
\begin{equation}\label{eq:modulated-energy}
 \cU_{N,\rho}(g;X^{1:N})
 :=\frac1{N-1}\sum_{1\le i<j\le N}g_\rho^\circ(X_i,X_j)
\end{equation}
and call it the modulated energy of \(X^{1:N}\) relative to \(\rho\).  The
cancellation \eqref{eq:canonical-centering} also identifies
\(\cU_{N,\rho}(g;X^{1:N})\) as a scaled canonical \(U\)-statistic with kernel
\(g_\rho^\circ\).  Expanding the centered kernel gives
\begin{equation}\label{eq:modulated-energy-expansion}
 \cU_{N,\rho}(g;X^{1:N})
 =\frac1{2(N-1)}\sum_{i=1}^N\sum_{j\ne i}g(X_i-X_j)
  -\sum_{i=1}^N(g*\rho)(X_i)
  +\frac N2\iint g(x-y)\rho(x)\rho(y)\dd x\dd y.
\end{equation}
We omit \(X^{1:N}\) when the configuration is clear.  The factor
\(1/(N-1)\) places the fluctuations at order one. 

The corresponding centered partition function is the normalizing constant of
the Gibbs tilt of \(\rho^{\otimes N}\):
\[
 \cZ_{N,\rho}(\beta;g)
 :=\E_{\rho^{\otimes N}}e^{-\beta\cU_{N,\rho}(g;X^{1:N})},
 \qquad \beta\ge0.
\]
By \eqref{eq:canonical-centering},
\(\E_{\rho^{\otimes N}}\cU_{N,\rho}(g;X^{1:N})=0\). 

For comparison with the standard convention, our normalization and the
traditional \(1/N\) pair normalization satisfy the exact finite-\(N\)
relations
\begin{equation}\label{eq:normalization-scaling}
	\begin{aligned}
		\frac1N\sum_{1\le i<j\le N}g_\rho^\circ(X_i,X_j)
		&=\frac{N-1}{N}\cU_{N,\rho}(g),\\
		\E_{\rho^{\otimes N}}\exp\!\Bigg(
		-\frac\beta N\sum_{1\le i<j\le N}g_\rho^\circ(X_i,X_j)\Bigg)
		&=\cZ_{N,\rho}\!\left(\frac{N-1}{N}\beta;g\right).
	\end{aligned}
\end{equation}
Thus the two conventions differ only by the displayed finite-\(N\) rescaling
of \(\beta\).

The basic problem is whether \(\cZ_{N,\rho}(\beta;g)\) remains bounded as
\(N\to\infty\) for every fixed \(\beta>0\).  This is substantially stronger
than weak convergence of the canonical \(U\)-statistic: convergence in law
does not control its exponential moments.  For singular kernels, existing
perturbative arguments also do not reach arbitrary \(\beta\).  A
finite-particle, all-temperature bound is therefore very crucial for four
connected questions: convergence of the centered partition functions,
identification of their determinant limit, quantitative control of the
associated Gibbs measures, and applications to mean-field dynamics and
fluctuations.

We prove such a bound for logarithmic, Riesz, and Bessel--Riesz kernels in the
full locally square-integrable range \(0\le s<d/2\).  This
finite-particle estimate yields uniform Gibbs entropy bounds, quantitative
entropic chaos, and an \(N^{-1/2}\) Gaussian fluctuation estimate. We also identify the
large-\(N\) limit as the Laplace transform of a second Wiener chaos random
variable and determine its sharp growth as \(\beta\to\infty\).  The latter
question is spectral: it depends on the number and distribution of the
eigenvalues of \(A_{g,\rho}\) defined in \eqref{eq:newOperator} above the scale \(\beta^{-1}\).

We now specify the singular interactions considered in the article.  We use the
Fourier convention
\(\widehat f(\xi)=\int_{\mathbb R^d}e^{-ix\cdot\xi}f(x)\dd x\).
For \(0\le s<d/2\), let
\[
 g_{d,s}(x)=
 \begin{cases}
  -\dfrac{2(4\pi)^{-d/2}}{\Gamma(d/2)}\log|x|,&s=0,\\
  c_{d,s}|x|^{-s},&0<s<d/2,
 \end{cases}
\]
where \(c_{d,s}>0\) is chosen so that
\(\widehat{g_{d,s}}(\xi)=|\xi|^{-(d-s)}\).  For \(m>0\) and
\(0<s<d/2\), define the Bessel--Riesz kernel by
\[
 \widehat{g_{d,s}^{(m)}}(\xi)
 =(m^2+|\xi|^2)^{-(d-s)/2},
\]
and set \(g_{d,s}^{(0)}=g_{d,s}\).  The Coulomb kernels in this range include
\[
 g_{2,0}(x)=-\frac1{2\pi}\log|x|,
 \qquad
 g_{3,1}(x)=\frac1{4\pi|x|},
 \qquad
 -\Delta g_{2,0}=\delta_0,
 \quad
 -\Delta g_{3,1}=\delta_0.
\]

The restriction \(s<d/2\) is exactly the range in which these kernels are
locally square integrable.  Under the following standing assumption, this
implies \(g_\rho^\circ\in L^2(\rho^{\otimes2})\).  Consequently, the centered
interaction operator is Hilbert--Schmidt and the Gaussian quadratic form that
appears in the limiting theory is well defined.

\begin{assumption}
\label{ass:kernel-density}
Let \(d\ge1\) and \(0\le s<d/2\).  Assume
\[
 \rho\ge0,
 \qquad
 \int_{\mathbb R^d}\rho(x)\dd x=1,
 \qquad
 \rho\in L^\infty(\mathbb R^d).
\]
When \(s=0\), take \(m=0\), corresponding to the repulsive logarithmic
potential.  We further assume that, for some \(a_\rho,q_\rho>0\),
\begin{equation}\label{eq:log-tail-assumption}
 \int_{\mathbb R^d}e^{a_\rho|x|^{q_\rho}}\rho(x)\dd x<\infty.
\end{equation}
If \(0<s<d/2\), let \(m\ge0\); no further assumption is needed.
\end{assumption}

For \(g=g_{d,s}^{(m)}\), define the centered interaction operator
\(A_{g,\rho}:L^2(\rho)\to L^2(\rho)\) by
\begin{equation}\label{eq:newOperator}
 A_{g,\rho}f(x)
 :=\int_{\mathbb R^d}g_\rho^\circ(x,y)f(y)\rho(y)\dd y.
\end{equation}

\subsection{All-temperature, uniform-in-\(N\) partition-function estimates}

The first main result resolves the finite-particle exponential-integrability
problem at every temperature.

\begin{theorem}
\label{thm:uniform-partition}
Under Assumption~\ref{ass:kernel-density}, for every \(\beta\ge0\),
\begin{equation}\label{eq:riesz-uniform}
 1\le\sup_{N\ge2}
 \cZ_{N,\rho}(\beta;g_{d,s}^{(m)})<\infty.
\end{equation}
More precisely, for \(0\le\beta\le1\),
\[
 0\le\sup_{N\ge2}
 \log\cZ_{N,\rho}(\beta;g_{d,s}^{(m)})
 \le C_{d,s,m,\rho}\beta^2,
\]
and for \(\beta\ge1\),
\begin{equation}\label{eq:riesz-parameter-bound}
 \sup_{N\ge2}\log\cZ_{N,\rho}(\beta;g_{d,s}^{(m)})
 \le C_{d,s,m,\rho}
 \begin{cases}
  \beta\log(2+\beta),&s=0,\\
  \beta^{d/(d-s)},&0<s<d/2.
 \end{cases}
\end{equation}
\end{theorem}

The estimate contains two distinct regimes.  Near \(\beta=0\), the canonical
cancellation removes the linear term and yields the quadratic bound
\(O(\beta^2)\).  At large \(\beta\), the logarithmic and Riesz singularities
produce the kernel-dependent orders in
\eqref{eq:riesz-parameter-bound}.  Besides being the principal result of the
article, Theorem~\ref{thm:uniform-partition} supplies the uniform integrability
needed for the determinant limit and uniform \(L^p\) control of the centered
Gibbs density.

The theorem extends the closest finite-particle estimates in three directions:
the singularity of the interaction, the temperature range, and the underlying
space.  Jabin and Wang proved a uniform partition-function estimate for a
bounded two-variable potential under marginal cancellation and a quantitative
smallness condition \cite[Theorem~4]{JabinWang2018}.  On the torus,
Duerinckx and Jabin established a uniform bound and the associated determinant
limit for centered square-integrable interactions in the small-\(\beta\)
regime \cite[Theorem~2.1 and Proposition~2.2]{DuerinckxJabin2026}.
Delgadino and Gvalani obtained an all-\(\beta\) estimate for repulsive
logarithmic-type interactions, with bounded reference density and the \(1/N\)
pair normalization \cite[Theorem~2.5]{DelgadinoGvalani2025}.  In contrast,
Theorem~\ref{thm:uniform-partition} treats probability densities on
\(\mathbb R^d\), every \(\beta\ge0\), and logarithmic, Riesz, and
Bessel--Riesz kernels throughout \(0\le s<d/2\).  It also indicates the explicit
large-\(\beta\) dependence; equation \eqref{eq:normalization-scaling} accounts
for the normalization difference.

\begin{remark}
The range \(s<d/2\) is sharp for an \(N\)-uniform estimate at this level of
generality.  Indeed, for \(d/2\le s<d\), the centered periodic Riesz kernel on
\(\mathbb T^d\) with Fourier coefficients
\(|k|^{s-d}\mathbf1_{k\ne0}\) does not belong to
\(L^2(\mathbb T^d)\), and Duerinckx and Jabin
\cite[Proposition~2.2(ii)]{DuerinckxJabin2026} prove divergence of its
\(1/N\)-normalized partition function for every positive parameter.  Since
centered Laplace transforms are nondecreasing on \([0,\infty)\), for
\(0<\beta'<\beta\),
\[
 \E_{\rho^{\otimes N}}\exp\!\Bigg(
 -\frac{\beta'}N\sum_{1\le i<j\le N}g_\rho^\circ(X_i,X_j)\Bigg)
 =\cZ_{N,\rho}\!\left(\frac{N-1}{N}\beta';g\right)
 \le \cZ_{N,\rho}(\beta;g)
\]
for all sufficiently large \(N\).  The same divergence therefore holds under
the \(1/(N-1)\) normalization, and \eqref{eq:riesz-uniform} cannot hold in
general when \(s\ge d/2\).  This argument does not rule out \(N\)-dependent
bounds; see \cite{DuerinckxJabin2026}. One can easily obtain an $N$-dependent bounds for Coulomb potentials in 
$\mathbb{R}^d$ with $d \geq 4$, by following our proof in Section \ref{sec:partition-estimates}. 
\end{remark}

The proof of Theorem~\ref{thm:uniform-partition} has an abstract combinatorial
part and a kernel-specific analytic part.  First, a Mayer--Penrose expansion
reduces the exponential moment of a marginally centered kernel to connected
graphs.  The cancellation improves the contribution of a tree by one power
of \(N\), and graphs with additional edges reduce to a single-cycle estimate.
We formulate this argument on an arbitrary probability space so that it can
later be applied to non-translation-invariant commutator kernels.

Second, a heat-kernel decomposition separates the singular interaction as
\(g=g_\tau+r_\tau\).  The coarse part \((g_\tau)_\rho^\circ\) is positive
semidefinite and has an explicitly controlled diagonal.  The remainder is
estimated in \(L^1\) and \(L^2\), with square integrability holding exactly
for \(s<d/2\).  Optimizing the cutoff \(\tau\) as a function of \(\beta\)
gives the large-\(\beta\) bounds in
\eqref{eq:riesz-parameter-bound}.  Near \(\beta=0\), canonical cancellation,
a uniform second moment, and one fixed exponential moment yield the quadratic
estimate.

\subsection{Three applications}

We now present three major applications of our uniform partition-function estimates, i.e. Theorem \ref{thm:uniform-partition}. 

\subsubsection{Application I: Entropic mean-field convergence near thermal
equilibrium}

We first present a simple consequence of a uniform partition-function
estimate, namely \eqref{eq:riesz-uniform}.  Let
\(g=g_{d,s}^{(m)}\), and let \(\rho\) satisfy
Assumption~\ref{ass:kernel-density}.  For \(\beta>0\), define the Gibbs tilt
\(\cP_{N,\rho}^{\beta}\) of \(\rho^{\otimes N}\) by
\begin{equation}\label{eq:centered-gibbs-measure}
 \dd\cP_{N,\rho}^{\beta}(X^{1:N})
 :=\frac{1}{\cZ_{N,\rho}(\beta;g)}
 e^{-\beta\cU_{N,\rho}(g;X^{1:N})}\dd\rho^{\otimes N}(X^{1:N}).
\end{equation}
Since
\(\E_{\rho^{\otimes N}}\cU_{N,\rho}(g;X^{1:N})=0\), adding the two relative
entropies gives
\[
 \Ent(\cP_{N,\rho}^{\beta}\mid\rho^{\otimes N})
 +\Ent(\rho^{\otimes N}\mid\cP_{N,\rho}^{\beta})
 =-\beta\E_{\cP_{N,\rho}^{\beta}}\cU_{N,\rho}(g;X^{1:N}).
\]
Applying the Gibbs variational inequality \cite[Lemma~1]{JabinWang2018}, we obtain
\[
 -\beta\E_{\cP_{N,\rho}^{\beta}}\cU_{N,\rho}(g;X^{1:N})
 \le\frac12\Ent(\cP_{N,\rho}^{\beta}\mid\rho^{\otimes N})
   +\frac12\log\cZ_{N,\rho}(2\beta;g).
\]
Since both relative entropies are nonnegative,
\[
 \Ent(\cP_{N,\rho}^{\beta}\mid\rho^{\otimes N})
 \le\log\cZ_{N,\rho}(2\beta;g)
\]
and the right-hand side is bounded uniformly in \(N\).  Equivalently, the
relative entropy per particle is
\(O(N^{-1})\).  An earlier argument of Cai--Feng--Gong--Wang
\cite{CFGW24}, based on Serfaty's pointwise lower bound for the modulated
energy \cite{Serfaty2020}, gives an \(N\)- and dimension-dependent estimate
for Coulomb and Riesz potentials.

More generally, Proposition~\ref{prop:gibbs-entropy} shows that, for every
fixed \(\beta>0\) and \(p>1\),
\begin{equation}\label{eq:centered-gibbs-renyi-entropy-bounds}
 D_p(\cP_{N,\rho}^{\beta}\mid\rho^{\otimes N})
 +\Ent(\cP_{N,\rho}^{\beta}\mid\rho^{\otimes N})
 +\Ent(\rho^{\otimes N}\mid\cP_{N,\rho}^{\beta})
 \le C_{g,\rho,\beta,p},
\end{equation}
uniformly in \(N\), where $D_p$ denotes the R\'enyi divergence of order \(p\)  as defined in Section \ref{subsec:notation}. Consequently, the R\'enyi divergence and both relative
entropies are bounded at the level of the full \(N\)-particle system, and the
normalized relative entropies are \(O(N^{-1})\).  This conclusion is genuinely
finite-particle: it follows directly from Theorem~\ref{thm:uniform-partition}
and does not rely on the limiting determinant.

This centered Gibbs structure also describes thermal equilibrium for a
second-order Coulomb--Riesz particle system.  Let \(V\in C^1(\mathbb R^d)\), let
\(\sigma\ge0\), and let \(B_1,\ldots,B_N\) be independent standard Brownian
motions in \(\mathbb R^d\).  We consider
\begin{equation}\label{eq:second-order-particle-system}
 \begin{cases}
  \dd X_i=V_i\dd t,\\
  \displaystyle
  \dd V_i=-\bigg(\nabla V(X_i)
   +\frac1{N-1}\sum_{j\ne i}\nabla g(X_i-X_j)\bigg)\dd t
   -\sigma V_i\dd t+\sqrt{\dfrac{2\sigma}{\beta}}\dd B_i,
 \end{cases}
 \qquad 1\le i\le N.
\end{equation}
Writing \(Z_i=(x_i,v_i)\) and
\(Z^{1:N}=(Z_1,\ldots,Z_N)\), the Hamiltonian reads 
\begin{equation*}
 H_{N,V}(Z^{1:N})
 =\frac12\sum_{i=1}^N\abs{v_i}^2+\sum_{i=1}^NV(x_i)
  +\frac1{2(N-1)}\sum_{i=1}^N\sum_{j\ne i}g(x_i-x_j).
\end{equation*}
Its \(N\)-particle Gibbs equilibrium is
\begin{equation*}
 \cG_{N,\beta}(Z^{1:N})
 :=\frac{e^{-\beta H_{N,V}(Z^{1:N})}}
 {\displaystyle\int e^{-\beta H_{N,V}}\dd Z^{1:N}}.
\end{equation*}

The formal mean-field limit of \eqref{eq:second-order-particle-system} is the
Vlasov--Fokker--Planck equation
\begin{equation}\label{eq:mean-field-vfp}
 \partial_t f+v\cdot\nabla_xf
 -\nabla_x\bigl(V+g*\rho_f\bigr)\cdot\nabla_vf
 =\sigma\nabla_v\!\cdot\left(vf+\frac1\beta\nabla_vf\right),
 \qquad
 \rho_f(t,x):=\int_{\mathbb R^d}f(t,x,v)\dd v.
\end{equation}
Suppose that a probability density \(\rho_\beta\), satisfying
Assumption~\ref{ass:kernel-density}, solves
\begin{equation*}
 \rho_\beta(x)=z_{\beta,V}^{-1}
 \exp\left(-\beta V(x)-\beta(g*\rho_\beta)(x)\right).
\end{equation*}
Here \(z_{\beta,V}>0\) is the normalizing constant.  Define
\begin{equation*}
 m_\beta(x,v):=\rho_\beta(x)
 \left(\frac{\beta}{2\pi}\right)^{d/2}e^{-\beta\abs v^2/2},
 \qquad
 \cM_{N,\beta}:=m_\beta^{\otimes N},
\end{equation*}
and \(m_\beta\) is stationary for \eqref{eq:mean-field-vfp}; when \(d=3\)
and \(g=g_{3,1}\), the equation~\eqref{eq:mean-field-vfp} is the
Vlasov--Poisson--Fokker--Planck equation.  The exact factorization proved in
Section~\ref{sec:gibbs-integrability} identifies the density of
\(\cG_{N,\beta}\) relative to \(\cM_{N,\beta}\) with the centered Gibbs
tilt.  Thus the bounds in \eqref{eq:centered-gibbs-renyi-entropy-bounds} on the
R\'enyi divergence and relative entropy
transfer from \((\cP_{N,\rho_\beta}^{\beta},\rho_\beta^{\otimes N})\) to
\((\cG_{N,\beta},\cM_{N,\beta})\).  Together with the
entropy-transfer inequality, this yields the following uniform estimate.  Let
\(F^N(t)\) denote the joint law at time \(t\) of the particle system
\eqref{eq:second-order-particle-system}.  Then
Proposition~\ref{prop:equilibrium-entropy-transfer} gives, for all \(N\ge2\)
and \(t\ge0\),
\[
 \Ent(F^N(t)\mid\cM_{N,\beta})
 \le 2\Ent(F^N(0)\mid\cG_{N,\beta})+C_{g,\rho_\beta,\beta},
\]
where the constant is independent of \(N\) and \(t\).

\subsubsection{Application II: Commutator estimates and quantitative
propagation of chaos}

The second application combines a uniform partition-function estimate with a
dynamic modulated free energy argument.
We consider the first-order mean-field diffusion
\begin{equation}\label{eq:intro-first-order-particle-system}
 \dd X_t^{N,i}
 =-\frac1{N-1}\sum_{j\ne i}\nabla g(X_t^{N,i}-X_t^{N,j})\dd t
   +\sqrt{\frac2\beta}\dd B_t^i,
 \qquad 1\le i\le N,
\end{equation}
whose mean-field density solves
\[
 \partial_t\rho_t
 -\nabla\!\cdot\!\bigl(\rho_t\nabla(g*\rho_t)\bigr)
 =\frac1\beta\Delta\rho_t.
\]
For the three-dimensional Coulomb flow, comparison of the commutator kernel
with the interaction kernel yields the exponential-moment estimate under the
centered Gibbs law in
Proposition~\ref{prop:gibbs-commutator-exponential}.
More explicitly, for a globally Lipschitz vector field \(u\), set
\[
 q_u(x,y):=(u(x)-u(y))\cdot\nabla g(x-y).
\]
For some \(\theta>0\), we find that
\[
 \sup_{N\ge2}\E_{\cP_{N,\rho}^{\beta}}
 \exp\!\left(\theta\abs{\cU_{N,\rho}(q_u)}\right)<\infty.
\]
The commutator estimate and the resulting propagation of chaos argument apply
to every kernel covered by Theorem~\ref{thm:uniform-partition}; the
three-dimensional Coulomb flow is singled out here only as the principal
example.
Entropy duality inserts
this estimate into the modulated free energy identity of
Proposition~\ref{prop:coulomb-modulated-identity}, and Gronwall's inequality
closes the argument.
The same partition-function bound supplies the coercive lower bound used to
recover relative entropy.
For product initial data,
Theorem~\ref{thm:optimal-modulated-free-energy} yields
\[
 \sup_{0\le t\le T}
 \Ent(F^N(t)\mid\rho_t^{\otimes N})\le C_T.
\]
Thus the total relative entropy remains bounded uniformly in \(N\), and the
normalized entropy has the optimal order \(N^{-1}\).
The modulated free energy converges uniformly to zero at the same
\(N^{-1}\) rate.
The key observation is that the commutator is controlled under the centered
Gibbs law rather than directly under the product law.

\subsubsection{Application III: Quantitative Gaussian fluctuations}

The third application uses a uniform partition-function estimate directly.
Section~\ref{sec:gaussian-fluctuations} treats the first-order
system~\eqref{eq:intro-first-order-particle-system} for every kernel covered by
Theorem~\ref{thm:uniform-partition}.
For the fluctuation measure
\(\eta_t^N:=\sqrt N\,\bigl(N^{-1}\sum_{i=1}^N
\delta_{X_t^{N,i}}-\rho_t\bigr)\), the equation
tested against a smooth function \(\varphi\) consists of a linearized drift, a
martingale, and the nonlinear remainder.
More precisely, applying It\^o's formula and symmetrization gives
\[
 \dd\langle\eta_t^N,\varphi\rangle
 =\langle\eta_t^N,\mathcal L_t\varphi\rangle\dd t
  -\frac1{\sqrt N}\cU_{N,\rho_t}(q_{\nabla\varphi})\dd t
  +\dd M_{N,t}^{\varphi},
\]
where \(M_{N,t}^{\varphi}\) denotes the martingale part and the linearized
operator is
\[
 (\mathcal L_t\varphi)(x)
 =\frac1\beta\Delta\varphi(x)
  -\nabla(g*\rho_t)(x)\cdot\nabla\varphi(x)
  +\int\nabla g(x-y)\cdot\nabla\varphi(y)\rho_t(y)\dd y.
\]
For the product initial data and regularity assumptions in
Section~\ref{sec:gaussian-fluctuations}, the modulated free energy inequality,
Theorem~\ref{thm:uniform-partition}, and
Proposition~\ref{prop:gibbs-commutator-exponential} imply
\[
 \sup_{N\ge2}\sup_{0\le t\le T}
 \Ent(F^N(t)\mid\cP_{N,\rho_t}^{\beta})<\infty,
 \qquad
 \sup_{N\ge2}\sup_{0\le t\le T}
 \E_{F^N(t)}\abs{\cU_{N,\rho_t}(q_{\nabla\varphi})}<\infty.
\]
The first estimate transfers concentration from the Gibbs reference law to
the particle law, while the second makes the nonlinear remainder
\(O(N^{-1/2})\) in \(L^1\).  Together with Hoeffding's inequality, the same
Gibbs entropy bound controls bounded empirical linear statistics and hence
the random part of the quadratic covariation, again with error
\(O(N^{-1/2})\) in \(L^1\).

At a fixed time, a backward equation for \(\mathcal L_t\) removes the
linearized drift.  We then compare the resulting finite-dimensional
semimartingale with its Gaussian limit.  The two preceding error bounds and a
multivariate Berry--Esseen estimate for the i.i.d.\ initial field give the
\(N^{-1/2}\) rate in
Theorem~\ref{thm:quantitative-gaussian-fluctuations}.
More precisely, for
smooth test functions
\(\boldsymbol\varphi=(\varphi_1,\ldots,\varphi_J)\), the theorem yields
\[
 \left|\E\Phi\!\left(\langle\eta_t^N,\boldsymbol\varphi\rangle\right)
       -\E\Phi(Y_t)\right|
 \le\frac{C_{t,\boldsymbol\varphi}}{\sqrt N}
 \left(\norm{\nabla\Phi}_\infty+\norm{\nabla^2\Phi}_\infty\right),
 \qquad \Phi\in C_b^2(\mathbb R^J),
\]
where \(Y_t\), depending on \(t\) and \(\boldsymbol\varphi\), is the centered
Gaussian vector determined by the Gaussian limit of the initial fluctuation
and the limiting martingale covariance.
In particular, every fixed-time finite-dimensional fluctuation vector
converges in law without a prior tightness or process-level convergence
argument for the full field.

\subsection{Large-\(N\) limit and sharp parameter dependence}

We next study the finer large-\(N\) asymptotics of the centered partition
function, identify its limit in terms of a Carleman--Fredholm determinant, and
use this limit to establish the sharpness of the parameter dependence in the
uniform bounds above.

\subsubsection{Second Wiener chaos and determinant limit}

The canonical \(U\)-statistic structure also determines the large-\(N\) limit
of the modulated energy.  The statistic converges weakly to a centered
quadratic form in a Gaussian field, whereas the partition function is its
Laplace transform.  The following theorem identifies both the fluctuation law
and its Laplace transform.

\begin{theorem}
\label{thm:gaussian-determinant}
Under Assumption~\ref{ass:kernel-density}, let \(g=g_{d,s}^{(m)}\).
Let \((\kappa_j)_{j\ge1}\) be the nonnegative eigenvalues of \(A_{g,\rho}\)
in nonincreasing order, repeated according to multiplicity, and let
\((Z_j)_{j\ge1}\) be independent standard Gaussian variables.  Then
\begin{equation}\label{eq:gaussian-chaos-limit}
 \cU_{N,\rho}(g)
 \xrightarrow[N\uparrow\infty]{\mathrm{law}}
 \cQ_{g,\rho}
 :=\frac12\sum_{j\ge1}\kappa_j(Z_j^2-1),
\end{equation}
where the series on the right-hand side converges in \(L^2\) and belongs to
the second Wiener chaos.
Moreover, for every \(\beta>0\),
\begin{align}
 \cZ_{\infty,\rho}(\beta;g)
 :=\lim_{N\uparrow\infty}\cZ_{N,\rho}(\beta;g)
 &=\E e^{-\beta\cQ_{g,\rho}}\notag\\
 &=\dettwo(\Id+\beta A_{g,\rho})^{-1/2},             \label{eq:determinant-limit}
\end{align}
and
\begin{equation}\label{eq:determinant-log}
 \lim_{N\uparrow\infty}\log\cZ_{N,\rho}(\beta;g)
 =\log\cZ_{\infty,\rho}(\beta;g)
 =\frac12\sum_{j\ge1}
 \bigl(\beta\kappa_j-\log(1+\beta\kappa_j)\bigr).
\end{equation}
\end{theorem}

The convergence in \eqref{eq:gaussian-chaos-limit} is the classical
second Wiener chaos limit for canonical \(U\)-statistics
\cite{DynkinMandelbaum1983,Lee2019}.  For fixed \(\beta>0\), applying
Theorem~\ref{thm:uniform-partition} at a slightly larger parameter gives
uniform integrability of \(e^{-\beta\cU_{N,\rho}(g)}\).  This upgrades
distributional convergence to convergence of the Laplace transforms and
yields \eqref{eq:determinant-limit}--\eqref{eq:determinant-log}.
Expanding \eqref{eq:determinant-log} at \(\beta=0\) yields
\[
 \log\cZ_{\infty,\rho}(\beta;g)
 =\frac{\beta^2}{4}\norm{g_\rho^\circ}_{L^2(\rho^{\otimes2})}^2+o(\beta^2),
 \qquad \beta\downarrow0.
\]
When \(g_\rho^\circ\ne0\), the positive quadratic coefficient shows that the
small-\(\beta\) order in Theorem~\ref{thm:uniform-partition} is sharp.

\subsubsection{Sharp dependence on \texorpdfstring{\(\beta\)}{beta}}

At large \(\beta\), the determinant formula turns the proof of the matching
lower bound into a spectral problem.

\begin{theorem}
\label{thm:sharp-coulomb-growth}
In addition to Assumption~\ref{ass:kernel-density}, assume that \(\rho\) is
bounded below by a positive constant almost everywhere on some ball
\(D\subset\mathbb R^d\).
If \(s=0\), there exist \(\beta_0\ge1\) and constants
\(c_{d,\rho,D},C_{d,\rho}\in(0,\infty)\) such that, for every
\(\beta\ge\beta_0\),
\begin{equation}\label{eq:sharp-log-growth}
 c_{d,\rho,D}\beta\log\beta
 \le\log\cZ_{\infty,\rho}(\beta;g_{d,0})
 \le\sup_{N\ge2}\log\cZ_{N,\rho}(\beta;g_{d,0})
 \le C_{d,\rho}\beta\log\beta.
\end{equation}
If \(0<s<d/2\), there exist \(\beta_0\ge1\) and constants
\(c_{d,s,m,\rho,D},C_{d,s,m,\rho}\in(0,\infty)\) such that, for every
\(\beta\ge\beta_0\),
\begin{equation}\label{eq:sharp-riesz-growth}
 c_{d,s,m,\rho,D}\beta^{d/(d-s)}
 \le\log\cZ_{\infty,\rho}(\beta;g_{d,s}^{(m)})
 \le\sup_{N\ge2}\log\cZ_{N,\rho}(\beta;g_{d,s}^{(m)})
 \le C_{d,s,m,\rho}\beta^{d/(d-s)}.
\end{equation}
\end{theorem}

For the lower bound, localized test functions supported where \(\rho\) is
bounded below, together with the max--min principle, give the required
eigenvalue count.  Combined with the finite-particle upper estimate in
\eqref{eq:riesz-parameter-bound}, this proves that the bounds have the correct
order as \(\beta\to\infty\).

The eigenvalues above the scale \(\beta^{-1}\) are the active modes in this
argument.  Their distribution and cumulative contribution explain the two
growth laws: in the logarithmic case they generate a harmonic sum and hence
the additional \(\log\beta\) factor, whereas in the Riesz case they yield the
power \(\beta^{d/(d-s)}\).

\subsection{Related works}

For regular interactions, mean-field Laplace expansions are classical
\cite{KusuokaTamura1984,MesserSpohn1982,Bolthausen1986,
BenArousBrunaud1990}, as are their stability under approximation
\cite{Guionnet1999}.  Canonical statistics converge to the second Wiener chaos
under general conditions \cite{DynkinMandelbaum1983}, while quantitative
non-Gaussian approximation results such as \cite{BentkusGotze1999} address
the corresponding distributional limit.  The present problem is different:
for singular interactions, we need exponential moments uniform in \(N\) and
valid beyond the perturbative temperature regime.  The comparison following
Theorem~\ref{thm:uniform-partition} describes the closest uniform
partition-function estimates for singular interactions, due to
Duerinckx--Jabin \cite{DuerinckxJabin2026} and Delgadino--Gvalani
\cite{DelgadinoGvalani2025}.

A complementary literature studies quantitative approximate independence and
concentration for regular mean-field Gibbs measures.  Lacker
\cite[Theorem~2.2]{Lacker2021} and Ren--Wang
\cite[Theorem~1]{RenWang2025} prove sharp local-chaos estimates under
transport and logarithmic-Sobolev assumptions, respectively.
Delgadino--Gvalani--Pavliotis--Smith
\cite[Theorems~3.6--3.7]{DelgadinoGvalaniPavliotisSmith2023} derive equilibrium
closeness and uniform-in-time propagation of chaos from a nondegenerate
particle logarithmic-Sobolev inequality.  Those results control marginals or
dynamical entropy under regularity or high-temperature hypotheses; our focus
is the centered partition function for singular kernels and the consequences
of its all-temperature bound.

Our fixed-\(\beta\), centered theory is also distinct from the microscopic
next-order theory of Coulomb and Riesz gases.  In the latter setting, the
parameter may vary with \(N\), and renormalized energies retain spatial
information that is absent from the determinant limit
\cite{RougerieSerfaty2016,PetracheSerfaty2017,LebleSerfaty2017,
Serfaty2023,Lewin2022,SerfatyLectures2024}.

For dynamics, we build on the relative-entropy and modulated free energy
framework of
\cite{JabinWang2018,BreschJabinWang2020,BreschJabinWang2023}.  In particular,
Proposition~\ref{prop:coulomb-modulated-identity} adapts the identity of
Bresch--Jabin--Wang \cite{BreschJabinWang2023}.  In the qualitative
fluctuation theory of Wang--Zhao--Zhu, entropy duality and exponential bounds
under the product law control the singular quadratic remainder, and a dual
backward equation identifies the Gaussian limit
\cite[Lemma~2.9, Theorem~1.4, and Proposition~1.5]{WangZhaoZhu2023}.
Hao--Zhang--Zhao use a related concentration mechanism for kinetic
McKean--Vlasov equations and obtain quantitative bounds for smooth tests
\cite[Section~4]{HaoZhangZhao2026}.

The distinction in our fluctuation argument is the reference law.  For
product initial data, a uniform partition-function estimate and the modulated
free energy inequality give a uniform entropy bound relative to the centered Gibbs
measure.  This single bound controls both empirical linear statistics and the
singular quadratic term, without first bounding the total entropy relative to
the product law.  We combine it with the finite-dimensional Gaussian
comparison from \cite[Section~4]{HaoZhangZhao2026} and a multivariate
Berry--Esseen estimate for the initial i.i.d.\ field.  The result is a direct
\(N^{-1/2}\) estimate for smooth finite-dimensional tests, independent of a
prior process-level convergence theorem.

M.~G.~Delgadino informed the first-named author that he, R.~S.~Gvalani, and
M.~Rosenzweig independently obtained a uniform partition-function estimate
of a similar form by different truncation methods \cite{Delgadino2026Personal}.

Our companion article \cite{WangZhao2026Attractive} treats the attractive
two-dimensional Coulomb interaction at the explicit log--HLS equilibrium
\(\rho_*(x)=[\pi(1+\abs{x}^2)^2]^{-1}\).  With the \(1/N\) pair normalization
in \eqref{eq:normalization-scaling} and \(T_*\) denoting the centered
interaction operator at \(\rho_*\), it proves uniform exponential
integrability and convergence of the attractive Laplace transform to
\(\dettwo(\Id-\beta T_*)^{-1/2}\) throughout the sharp range
\(0<\beta<8\pi\); indeed, for every fixed \(N\ge2\), the exponential moment
is finite if and only if \(\beta<8\pi\).  For the same normalization and
kernel \(-(2\pi)^{-1}\log\abs{x-y}\), Grotto's positive parameter \(\beta\)
is the attractive coupling, corresponding to physical inverse temperature
\(-\beta\), so the collision threshold is again \(8\pi\).  On
\(\mathbb T^2\), Grotto \cite[Theorem~1]{Grotto2026} proves the determinant
limit and Gaussian energy--enstrophy fluctuations for an unspecified range
\(0<\beta<\beta_0<8\pi\). 

\subsection{Outline of the article}

Section~\ref{sec:preliminary} introduces the notation and graph estimates.
Section~\ref{sec:partition-estimates} proves the abstract exponential estimate
and a uniform partition-function estimate.  The next three sections develop
applications: Section~\ref{sec:gibbs-integrability} treats Gibbs entropy and
thermal equilibrium; Section~\ref{sec:optimal-modulated-free-energy} proves
the commutator estimate and quantitative entropic chaos; and
Section~\ref{sec:gaussian-fluctuations} establishes the fixed-time
\(N^{-1/2}\) Gaussian fluctuation bound.  Section~\ref{sec:second-order-limits}
identifies the second Wiener chaos and determinant limits and determines the
sharp dependence on \(\beta\).

\subsection*{Acknowledgements}
Z.~Wang thanks Matias Delgadino for sharing updates on recent joint work
\cite{Delgadino2026Personal} and Pierre-Emmanuel Jabin for discussions about
positive- and negative-definite decompositions in uniform partition-function
estimates.  This work was partially supported by the National Key R\&D Program of China (Project
No.~2024YFA1015500) and the National Natural Science Foundation of China
(NSFC; Grant Nos.~12595282 and 12171009).

\section{Preliminaries}
\label{sec:preliminary}

This section collects the notation, kernel assumptions, and graph estimates
used below.

\subsection{Notation}
\label{subsec:notation}

We collect the conventions used throughout the article;
symbols used only in a particular proof are defined when first
introduced.  For a finite signed measure
\(\nu\) on a measurable space \(E\) and an integrable function \(f\), we write
\[
 \langle\nu,f\rangle:=\int_E f\dd\nu,
 \qquad
 \|\nu\|_{\TV}:=|\nu|(E).
\]
For probability measures \(P\) and \(R\) on the same measurable space, their
relative entropy is
\[
 \Ent(P\mid R)
 :=\begin{cases}
 \displaystyle\int\log\!\left(\frac{\dd P}{\dd R}\right)\dd P,
     &P\ll R,\\
 +\infty,&\text{otherwise}.
 \end{cases}
\]
For \(p>1\), their R\'enyi divergence of order \(p\) is
\[
 D_p(P\mid R)
 :=\frac1{p-1}\log\int
 \left(\frac{\dd P}{\dd R}\right)^p\dd R,
\]
again with value \(+\infty\) when \(P\not\ll R\).
Notice that \(\Ent(P\mid R)\le D_p(P\mid R)\).

If \(\mu\) is a probability measure on \(E\), then \(\mu^{\otimes S}\)
denotes the product measure indexed by a finite set \(S\), and
\(\mu^{\otimes N}\) is its \(N\)-fold product.  Unless stated otherwise,
\(X_1,\ldots,X_N\) are independent with law \(\mu\), and
\[
 X^{1:N}=(X_1,\ldots,X_N),\qquad X_S=(X_i)_{i\in S}.
\]
On \(\mathbb R^{dN}\), \(\dd X^{1:N}=\dd x_1\cdots\dd x_N\).  We write
\(\E_P\) for expectation under \(P\), and omit the subscript when the law is
clear.  Statements about kernels on \(E^2\) are understood
\(\mu^{\otimes2}\)-almost everywhere.  Unless another measure is specified,
essential suprema in one variable are taken with respect to \(\mu\).

For a measurable kernel \(K:E\times E\to\mathbb R\) and
\(p\in[1,\infty)\), set
\[
\begin{aligned}
 \norm{K}_{L_x^\infty L_y^p(\mu)}
 &:=\esssup_x\left(\int_E\abs{K(x,y)}^p\,\mu(\dd y)\right)^{1/p},\\
 \norm{K}_{L_y^\infty L_x^p(\mu)}
 &:=\esssup_y\left(\int_E\abs{K(x,y)}^p\,\mu(\dd x)\right)^{1/p}.
\end{aligned}
\]

For a measure \(\nu\), the spaces \(L^p(\nu)\) and their norms have their
usual meaning; when the measure or domain is clear, we abbreviate the norm by
\(\|\cdot\|_p\).  We identify a density \(\rho\) with the measure
\(\rho(x)\dd x\), so in particular \(L^p(\rho)=L^p(\rho(x)\dd x)\).
The space \(C_c^\infty\) consists of smooth
compactly supported functions, while \(C_b^k\) consists of functions whose
derivatives up to order \(k\) are continuous and bounded; its norm is the sum
of the corresponding supremum norms.  We use the standard Sobolev notation
\(W^{k,p}\) and \(H^k=W^{k,2}\).
We set
\[
 L_0^2(\rho)
 :=\left\{f\in L^2(\rho):
 \int_{\mathbb R^d}f(x)\rho(x)\dd x=0\right\}.
\]
For an integrable two-variable kernel \(K\),
its centering with respect to \(\rho\) is
\[
 K_\rho^\circ(x,y)
 :=K(x,y)-\int K(x,z)\rho(z)\dd z
            -\int K(z,y)\rho(z)\dd z
            +\iint K(z,z')\rho(z)\rho(z')\dd z\dd z'.
\]
Thus \(g_\rho^\circ\) denotes this centering for
\(K(x,y)=g(x-y)\).  We omit \(X^{1:N}\) from
\(\cU_{N,\rho}(g;X^{1:N})\) whenever the configuration is clear.

Let \(H_1,H_2\) be separable Hilbert spaces.  For a bounded operator
\(A:H_1\to H_2\), \(\|A\|_{\mathrm{op}}\) denotes the operator norm.  The
operator is Hilbert--Schmidt if, for one (equivalently, every) orthonormal
basis \((e_j)\) of \(H_1\),
\[
 \|A\|_{\mathrm{HS}}^2:=\sum_j\|Ae_j\|_{H_2}^2<\infty.
\]
Every Hilbert--Schmidt operator is compact.  Moreover, the integral operator
\[
 A_Kf(x):=\int_E K(x,y)f(y)\,\mu(\dd y)
\]
associated with \(K\in L^2(\mu^{\otimes2})\) is Hilbert--Schmidt, with
\(\|A_K\|_{\mathrm{HS}}=\|K\|_{L^2(\mu^{\otimes2})}\).  If \(A\) is a
self-adjoint Hilbert--Schmidt operator with nonzero eigenvalues
\((\kappa_j)\), counted with multiplicity, then
\[
 \|A\|_{\mathrm{HS}}^2=\sum_j\kappa_j^2.
\]
For finite matrices, \(\|\cdot\|_{\mathrm{HS}}\) is the Frobenius norm.  We
write \(\Id\) for the identity and \(u\otimes v\) for the rank-one operator
\(h\mapsto\langle v,h\rangle u\).  For a self-adjoint Hilbert--Schmidt
operator \(A\) with eigenvalues \((\kappa_j)\), the Carleman--Fredholm determinant
is denoted by
\[
 \dettwo(\Id+A):=\prod_j(1+\kappa_j)e^{-\kappa_j}.
\]

Finally, \(\mathbf1\) denotes the constant function equal to one or the
all-ones vector, according to context.  For \(u\in\mathbb R\), set
\(u_-=\max\{-u,0\}\).  Throughout, \(C\) denotes a finite positive constant
whose value may change from line to line; displayed subscripts indicate its
permitted dependencies.

\subsection{Marginal cancellation and graph estimates}

Let \(K:E\times E\to\mathbb R\) be measurable.  We impose the two marginal
cancellation conditions
\begin{equation}\label{eq:canonical-assumptions}
 \int_E K(x,y)\,\mu(\dd y)=0,
 \quad \mu\text{-a.e. }x,
 \qquad
 \int_E K(x,y)\,\mu(\dd x)=0,
 \quad \mu\text{-a.e. }y.
\end{equation}
We also assume that
\begin{equation}\label{eq:kernel-assumptions}
 \norm{K}_{L^2(\mu^{\otimes2})}
 +\norm{K_-}_{L^\infty(\mu^{\otimes2})}<\infty.
\end{equation}
In the terminology of \(U\)-statistics,
\eqref{eq:canonical-assumptions} means that \(K\) is \(\mu\)-canonical, or
completely degenerate.  For independent \(X_1,\ldots,X_N\) with law \(\mu\),
set
\[
 \cU_{N,\mu}(K;X^{1:N})
 :=\frac1{N-1}\sum_{1\le i<j\le N}K(X_i,X_j).
\]
When \(K\) is symmetric, this is the usual scaled canonical \(U\)-statistic.
With \(E=\mathbb R^d\), \(\mu(\dd x)=\rho(x)\dd x\), and
\(K=g_\rho^\circ\), identity \eqref{eq:modulated-energy} identifies it with
the modulated energy.

The proof of Theorem~\ref{thm:canonical} uses the Penrose
tree-graph identity \eqref{eq:penrose}.  For a finite ordered vertex set \(S\), let
\(\mathcal C(S)\) and \(\mathcal T(S)\) denote the connected graphs and the
trees with vertex set \(S\), respectively.  Fix a rooted Penrose partition
scheme.  For \(T\in\mathcal T(S)\), let \(R(T)\) be the associated set of
non-tree edges.  Then, for arbitrary real edge weights \(z_e\),
\begin{equation}\label{eq:penrose}
 \sum_{G\in\mathcal C(S)}
 \prod_{e\in E(G)}z_e
 =\sum_{T\in\mathcal T(S)}
 \left(\prod_{e\in E(T)}z_e\right)
 \left(\prod_{e\in R(T)}(1+z_e)\right).
\end{equation}
This identity is standard; see
\cite[Section~4.1]{FernandezProcacci2007} and
\cite[Definition~7.1, Example~7.2, and Equation~(7.1)]{Jansen2018}.

The following lemma provides the estimates used with \eqref{eq:penrose} in the
proof of Theorem~\ref{thm:canonical}.  Part~(ii) extends the corresponding
estimate of Duerinckx and Jabin
\cite[proof of Theorem~2.1(i), Section~2.1]{DuerinckxJabin2026}
to an arbitrary probability space.

\begin{lemma}\label{lem:tree-unicyclic}
Let \(F:E\times E\to\mathbb R\) be measurable and assume that
\[
 \max\left\{
  \norm{F}_{L_x^\infty L_y^1(\mu)},
  \norm{F}_{L_y^\infty L_x^1(\mu)}
 \right\}<\infty.
\]
For an edge \(e=\{i,j\}\) with \(i<j\), write
\(F_e=F(x_i,x_j)\).
\begin{enumerate}
\item[(i)] Suppose that
\begin{equation}\label{eq:positive-marginals}
 \int_E F(x,y)\,\mu(\dd y)\ge0
 \quad\text{for \(\mu\)-a.e. \(x\)},
 \qquad
 \int_E F(x,y)\,\mu(\dd x)\ge0
 \quad\text{for \(\mu\)-a.e. \(y\)}.
\end{equation}
If \(T\) is a tree on \(\{1,\ldots,n\}\), where \(n\ge2\), then
\[
 \left|\int_{E^n}\prod_{e\in E(T)}F_e\,
 \mu^{\otimes n}(\dd X^{1:n})\right|
 \le\left(\iint_{E^2}F(x,y)\,\mu(\dd x)\mu(\dd y)\right)
 \max\left\{
  \norm{F}_{L_x^\infty L_y^1(\mu)},
  \norm{F}_{L_y^\infty L_x^1(\mu)}
 \right\}^{n-2}.
\]

\item[(ii)] Suppose that \(F\in L^2(\mu^{\otimes2})\).  If
\(G\) is a connected graph on \(\{1,\ldots,n\}\) with exactly one cycle,
then
\[
 \int_{E^n}\prod_{e\in E(G)}|F_e|\,
 \mu^{\otimes n}(\dd X^{1:n})
 \le \norm{F}_{L^2(\mu^{\otimes2})}^2
 \max\left\{
  \norm{F}_{L_x^\infty L_y^1(\mu)},
  \norm{F}_{L_y^\infty L_x^1(\mu)}
 \right\}^{n-2}.
\]
\end{enumerate}
\end{lemma}

\begin{proof}
Write
\[
 M:=\max\left\{
  \norm{F}_{L_x^\infty L_y^1(\mu)},
  \norm{F}_{L_y^\infty L_x^1(\mu)}
 \right\}.
\]
We use the same elementary estimate in both parts.  If vertex \(i\) has degree one
and \(j\) is its unique neighbor, then, for \(e=\{i,j\}\),
\begin{equation}\label{eq:one-vertex-integration}
 \int_E|F_e|\,\mu(\dd x_i)\le M
 \qquad\text{for \(\mu\)-a.e. \(x_j\)}.
\end{equation}
Indeed, this is one of the two mixed norm bounds, according to the order of
\(i\) and \(j\).  Since \(x_i\) occurs only in \(F_e\), after taking absolute
values the integration in \(x_i\) is bounded by \(M\) times the integral over
the remaining variables.  Repeating \eqref{eq:one-vertex-integration} along a
tree shows that the corresponding product is absolutely integrable, so
Fubini's theorem may be used below.

For part~(i), choose a vertex \(i\) of degree one in \(T\), and let \(j\) be
its unique neighbor.  Suppose first that \(i<j\).  By the second inequality in
\eqref{eq:positive-marginals},
\[
 h_j(x_j):=\int_EF(x_i,x_j)\,\mu(\dd x_i)\ge0
 \qquad\text{for \(\mu\)-a.e. \(x_j\)}.
\]
We first integrate in \(x_i\) before taking absolute values.  After deleting
\(i\), repeatedly choose a vertex of degree one different from \(j\), and
integrate its variable using \eqref{eq:one-vertex-integration}.  After \(n-2\)
such steps, only \(j\) remains.  Therefore,
\[
\begin{aligned}
 \left|\int_{E^n}\prod_{e\in E(T)}F_e\,
        \mu^{\otimes n}(\dd X^{1:n})\right|
 &\le \int_{E^{n-1}}h_j(x_j)
       \prod_{e\in E(T)\setminus\{\{i,j\}\}}|F_e|
       \prod_{k\ne i}\mu(\dd x_k)\\
 &\le M^{n-2}\int_Eh_j(x_j)\,\mu(\dd x_j)\\
 &=M^{n-2}\iint_{E^2}F(x,y)\,\mu(\dd x)\mu(\dd y).
\end{aligned}
\]
If \(j<i\), we instead use
\(h_j(x_j)=\int_EF(x_j,x_i)\,\mu(\dd x_i)\), which is nonnegative by the
first inequality in \eqref{eq:positive-marginals}; the calculation is the
same.  This proves part~(i).

For part~(ii), replace \(F\) by \(|F|\), which changes neither \(M\) nor the
\(L^2(\mu^{\otimes2})\) norm.  We may therefore assume that \(F\ge0\).  Let
\(v_1,\ldots,v_r\) be the vertices of the unique cycle of \(G\), listed in
cyclic order, and set \(v_{r+1}=v_1\).  Every other vertex belongs to a tree
that meets the cycle at one vertex.  Repeatedly choose a vertex of degree one
outside \(\{v_1,\ldots,v_r\}\) in the remaining graph and integrate its
variable using \eqref{eq:one-vertex-integration}.  After these \(n-r\) vertices
have been removed, we obtain
\[
 \int_{E^n}\prod_{e\in E(G)}F_e\,
 \mu^{\otimes n}(\dd X^{1:n})
 \le M^{n-r}
 \int_{E^r}\prod_{k=1}^rF_{\{v_k,v_{k+1}\}}
 \prod_{k=1}^r\mu(\dd x_{v_k}).
\]

It remains to estimate the integral over the cycle.  For \(R>0\), set
\(F_R=F\wedge R\), and let \(A_R\) be the integral operator with kernel
\(F_R\).  For \(h\in L^2(\mu)\),
\begin{align*}
 \norm{A_Rh}_{L^2(\mu)}^2
 &\le\int_E\left(\int_EF_R(x,y)\,\mu(\dd y)\right)
       \left(\int_EF_R(x,y)|h(y)|^2\,\mu(\dd y)\right)\mu(\dd x)\\
 &\le M^2\norm{h}_{L^2(\mu)}^2.
\end{align*}
The same calculation with the two variables interchanged gives
\[
 \norm{A_R}_{\mathrm{op}},\ \norm{A_R^*}_{\mathrm{op}}\le M,
 \qquad
 \norm{A_R}_{\mathrm{HS}}=\norm{A_R^*}_{\mathrm{HS}}
 =\norm{F_R}_{L^2(\mu^{\otimes2})}.
\]
Depending on the order of the two indices in each factor, the corresponding
cycle integral with \(F_R\) in place of \(F\) is
\(\operatorname{Tr}(B_1\cdots B_r)\), where every \(B_k\) is \(A_R\) or
\(A_R^*\).  By H\"older's inequality for Schatten norms,
\[
 \left|\operatorname{Tr}(B_1\cdots B_r)\right|
 \le \norm{B_1}_{\mathrm{HS}}\norm{B_2}_{\mathrm{HS}}
      \prod_{k=3}^r\norm{B_k}_{\mathrm{op}}
 \le \norm{F_R}_{L^2(\mu^{\otimes2})}^2M^{r-2}.
\]
By monotone convergence as \(R\uparrow\infty\), the cycle integral is at most
\(\norm{F}_{L^2(\mu^{\otimes2})}^2M^{r-2}\).  Combining this with the
\(n-r\) integrations outside the cycle gives
\[
 M^{n-r}\norm{F}_{L^2(\mu^{\otimes2})}^2M^{r-2}
 =\norm{F}_{L^2(\mu^{\otimes2})}^2M^{n-2},
\]
which proves part~(ii).
\end{proof}

\section{Uniform partition-function estimates}
\label{sec:partition-estimates}

In this section, we prove Theorem~\ref{thm:uniform-partition}.
In Section~\ref{subsec:general-exponential-estimate}, we prove a general
exponential estimate under marginal cancellation.  In
Section~\ref{subsec:heat-decomposition}, we establish the heat kernel
estimates for \(g_{d,s}^{(m)}\), which are combined in
Section~\ref{subsec:proof-uniform-partition} to complete the proof.

\subsection{General exponential estimate}
\label{subsec:general-exponential-estimate}

The following theorem establishes a uniform partition-function estimate for
every kernel satisfying the marginal cancellation and integrability assumptions
\eqref{eq:canonical-assumptions}--\eqref{eq:kernel-assumptions}.
Its proof follows the Mayer--Penrose argument of Duerinckx and Jabin; see
\cite[Theorem~2.1(i), Section~2.1, and Theorem~A.3]{DuerinckxJabin2026}.
In their centered formulation, the term associated with a tree vanishes.
Here the corresponding marginal integrals are only known to be nonnegative,
so Lemma~\ref{lem:tree-unicyclic}(i) estimates that term directly, while
part~(ii) treats the terms obtained by adding one edge to the tree.  The same
argument applies to nonsymmetric kernels on an arbitrary probability space.

\begin{theorem}
\label{thm:canonical}
Under the assumptions
\eqref{eq:canonical-assumptions}--\eqref{eq:kernel-assumptions}, there
exist universal constants \(c,C>0\) such that, for every
\(a>0\) satisfying
\begin{equation}\label{eq:canonical-smallness}
 a\norm{K_-}_{L^\infty(\mu^{\otimes2})}\le c,
\end{equation}
one has, for every \(N\ge2\),
\begin{equation}\label{eq:canonical-bound}
 1\le\E e^{-a\cU_{N,\mu}(K)}
 \le\exp\left(
 Ca^2\norm{K}_{L^2(\mu^{\otimes2})}^2
 \right).
\end{equation}
\end{theorem}

\begin{proof}
The cancellations in \eqref{eq:canonical-assumptions} imply
\(\E\,\cU_{N,\mu}(K)=0\).  By Jensen's inequality, we obtain the lower bound in
\eqref{eq:canonical-bound}.  It remains to prove the upper bound.

We first write the partition function as a finite product and expand it into
graphs.  After the graphs are grouped by their connected
components, the Penrose identity reduces the proof to the two estimates in
Lemma~\ref{lem:tree-unicyclic}.

Fix \(N\ge2\) and define
\[
 f_N(x,y)=\exp\left(-\frac{a}{N-1}K(x,y)\right)-1.
\]
Then
\begin{equation}\label{eq:pair-product}
 e^{-a\cU_{N,\mu}(K;X^{1:N})}
 =\prod_{1\le i<j\le N}\bigl(1+f_N(X_i,X_j)\bigr).
\end{equation}
The pointwise bound
\[
 -1\le f_N\le
 e^{a\norm{K_-}_{L^\infty(\mu^{\otimes2})}/(N-1)}-1
\]
shows that \(f_N\) is bounded.  Thus every term in the finite expansion of
\eqref{eq:pair-product} is integrable, and termwise integration gives
\begin{equation}\label{eq:graph-expansion}
 \E e^{-a\cU_{N,\mu}(K)}
 =\sum_{\substack{G\text{ a graph}\\V(G)=\{1,\ldots,N\}}}
   \int_{E^N}
   \prod_{\substack{\{i,j\}\in E(G)\\i<j}}f_N(x_i,x_j)
   \,\mu^{\otimes N}(\dd X^{1:N}).
\end{equation}

We now group the graphs in \eqref{eq:graph-expansion} by their connected
components.  For \(S\subset\{1,\ldots,N\}\) with \(\abs S\ge2\), set
\[
 \zeta_N(S)=
 \sum_{\substack{G\text{ connected}\\V(G)=S}}
 \int_{E^S}\prod_{\substack{\{i,j\}\in E(G)\\i<j}}f_N(x_i,x_j)
 \,\mu^{\otimes S}(\dd X_S).
\]
Every graph in \eqref{eq:graph-expansion} has a unique collection
\(\mathcal F\) of connected components containing at least one edge.  The
sets in \(\mathcal F\) are pairwise disjoint, no edge joins two different
sets, and the remaining vertices are isolated.  Hence Fubini's theorem expresses
its integral as the product of the integrals associated with the sets in
\(\mathcal F\); an isolated vertex contributes \(1\).  Conversely, choosing
such a family \(\mathcal F\) and a connected graph on every \(S\in\mathcal F\)
determines exactly one graph in \eqref{eq:graph-expansion}.  Therefore
\begin{equation}\label{eq:connected-expansion}
 \E e^{-a\cU_{N,\mu}(K)}
 =\sum_{\mathcal F}\prod_{S\in\mathcal F}\zeta_N(S).
\end{equation}
Here \(\mathcal F\) ranges over all families of pairwise disjoint subsets of
\(\{1,\ldots,N\}\) of size at least two, including the empty family.  This is
the standard connected-component decomposition; see
\cite[proof of Theorem~4.12, Equation~(4.12)]{Jansen2018}.

After taking absolute values in \eqref{eq:connected-expansion} and dropping
the disjointness restriction, we obtain the product over all subsets \(S\).  Using
\(1+t\le e^t\), we obtain
\begin{equation}\label{eq:connected-upper-bound}
 \E e^{-a\cU_{N,\mu}(K)}
 \le\prod_{\abs S\ge2}(1+\abs{\zeta_N(S)})
 \le\exp\left(\sum_{\abs S\ge2}\abs{\zeta_N(S)}\right).
\end{equation}
It remains to prove
\begin{equation}\label{eq:zeta-target}
 \sum_{\abs S\ge2}\abs{\zeta_N(S)}
 \le Ca^2\norm{K}_{L^2(\mu^{\otimes2})}^2.
\end{equation}

Fix \(S\) with \(\abs S=n\).  For \(e=\{i,j\}\) with \(i<j\), write
\(f_e=f_N(x_i,x_j)\).  By the Penrose identity \eqref{eq:penrose},
\begin{equation}\label{eq:zeta-penrose}
 \zeta_N(S)
 =\sum_{T\in\mathcal T(S)}\int
 \left(\prod_{e\in E(T)}f_e\right)
 \left(\prod_{e\in R(T)}(1+f_e)\right)
 \mu^{\otimes S}(\dd X_S).
\end{equation}
Writing
\(\prod_{e\in R(T)}(1+f_e)
=1+[\prod_{e\in R(T)}(1+f_e)-1]\) in
\eqref{eq:zeta-penrose}, we obtain
\begin{align}
 \zeta_N(S)
 &=\sum_{T\in\mathcal T(S)}
   \int\prod_{e\in E(T)}f_e\,\mu^{\otimes S}(\dd X_S)\notag\\
 &\quad+\sum_{T\in\mathcal T(S)}\int
   \left(\prod_{e\in E(T)}f_e\right)
   \left[\prod_{e\in R(T)}(1+f_e)-1\right]
   \mu^{\otimes S}(\dd X_S).
 \label{eq:zeta-two-term-split}
\end{align}
The first sum in \eqref{eq:zeta-two-term-split} contains one integral for each
tree.  In the second sum, we expand the expression in brackets and select one
additional edge.  The selected edge and the unique path between its endpoints
in the tree form exactly one cycle.  We estimate the two sums in this order,
using parts~(i) and~(ii) of Lemma~\ref{lem:tree-unicyclic}, respectively.

Both estimates use
\[
 \ell_N:=\max\left\{
 \norm{f_N}_{L_x^\infty L_y^1(\mu)},
 \norm{f_N}_{L_y^\infty L_x^1(\mu)}
 \right\}.
\]
By the first cancellation in
\eqref{eq:canonical-assumptions}, for \(\mu\)-almost every \(x\), we obtain
\[
 \int_E|K(x,y)|\,\mu(\dd y)
 =2\int_EK_-(x,y)\,\mu(\dd y)
 \le2\norm{K_-}_{L^\infty(\mu^{\otimes2})}.
\]
By the second cancellation, we obtain the same estimate after interchanging \(x\)
and \(y\).  Therefore
\begin{equation}\label{eq:K-marginal-L1}
 \max\left\{
  \norm{K}_{L_x^\infty L_y^1(\mu)},
  \norm{K}_{L_y^\infty L_x^1(\mu)}
 \right\}
 \le2\norm{K_-}_{L^\infty(\mu^{\otimes2})}.
\end{equation}
We shall also use the elementary inequalities
\begin{equation}\label{eq:elementary-exponential-bounds}
 \abs{e^{-u}-1}\le\abs u e^{u_-},
 \qquad
 0\le e^{-u}-1+u\le\frac12u^2e^{u_-},
 \qquad u\in\mathbb R.
\end{equation}
Choose \(c\le1\) in \eqref{eq:canonical-smallness}.  By the first inequality in
\eqref{eq:elementary-exponential-bounds} and
\eqref{eq:K-marginal-L1}, we obtain
\begin{equation}
 \ell_N
 \le\frac{Ca}{N}\norm{K_-}_{L^\infty(\mu^{\otimes2})}.
 \label{eq:ellN-bound}
\end{equation}

We begin with the first sum in \eqref{eq:zeta-two-term-split}.  To apply
Lemma~\ref{lem:tree-unicyclic}(i), we verify its two marginal conditions.  By
Jensen's inequality and the first cancellation in
\eqref{eq:canonical-assumptions}, for \(\mu\)-almost every \(x\),
\[
 \int_E f_N(x,y)\,\mu(\dd y)
 \ge \exp\left(-\frac{a}{N-1}
       \int_EK(x,y)\,\mu(\dd y)\right)-1=0.
\]
By the second cancellation, the same calculation gives
\[
 \int_E f_N(x,y)\,\mu(\dd x)\ge0
 \quad\text{for \(\mu\)-almost every \(y\)}.
\]
In particular, \(\iint f_N\dd\mu^{\otimes2}\ge0\).  Moreover,
\eqref{eq:canonical-assumptions} implies
\(\iint K\dd\mu^{\otimes2}=0\).
Set \(u=aK/(N-1)\).  By the second inequality in
\eqref{eq:elementary-exponential-bounds},
\begin{equation}\label{eq:fN-mean-bound}
 0\le\iint_{E^2}f_N\dd\mu^{\otimes2}
 =\iint_{E^2}\bigl(e^{-u}-1+u\bigr)\dd\mu^{\otimes2}
 \le\frac{Ca^2}{N^2}\norm{K}_{L^2(\mu^{\otimes2})}^2.
\end{equation}

Applying Lemma~\ref{lem:tree-unicyclic}(i) to \(F=f_N\), we obtain, for every
tree \(T\in\mathcal T(S)\),
\begin{equation}\label{eq:tree-integral-bound}
 \left|\int_{E^S}\prod_{e\in E(T)}f_e\,
 \mu^{\otimes S}(\dd X_S)\right|
 \le\left(\iint_{E^2}f_N(x,y)\,\mu(\dd x)\mu(\dd y)\right)
 \ell_N^{n-2}.
\end{equation}

We next estimate the second sum in \eqref{eq:zeta-two-term-split}.  Part~(ii)
of Lemma~\ref{lem:tree-unicyclic} also requires an \(L^2\) bound.
By the first inequality in \eqref{eq:elementary-exponential-bounds},
\begin{equation}
 \norm{f_N}_{L^2(\mu^{\otimes2})}^2
 \le\frac{Ca^2}{N^2}\norm{K}_{L^2(\mu^{\otimes2})}^2.
 \label{eq:fN-L2-bound}
\end{equation}
Fix an ordering of \(R(T)\).  Then
\begin{equation}\label{eq:tree-remainder-expansion}
 \prod_{e\in R(T)}(1+f_e)-1
 =\sum_{e'\in R(T)}f_{e'}
   \prod_{\substack{e\in R(T)\\e<e'}}(1+f_e).
\end{equation}
By the smallness condition
\eqref{eq:canonical-smallness} and
\(\abs{R(T)}\le\binom n2\le n(N-1)/2\), we obtain
\[
 \frac{a\norm{K_-}_{L^\infty(\mu^{\otimes2})}\abs{R(T)}}{N-1}
 \le \frac{a\norm{K_-}_{L^\infty(\mu^{\otimes2})}n}{2}
 \le \frac n2.
\]
By this bound, the identity
\(1+f_e=\exp\bigl(-aK(x_i,x_j)/(N-1)\bigr)>0\), and \(-K\le K_-\),
we obtain
\begin{equation}\label{eq:tree-positive-factor-bound}
 \prod_{\substack{e\in R(T)\\e<e'}}(1+f_e)
 \le \exp\left(
  \frac{a\norm{K_-}_{L^\infty(\mu^{\otimes2})}\abs{R(T)}}{N-1}
  \right)
 \le C^n.
\end{equation}
A tree has a unique path between the endpoints of \(e'\).  This path together
with \(e'\) forms exactly one cycle, and all other tree edges lie outside that
cycle.  Hence the graph with edge set \(E(T)\cup\{e'\}\) satisfies the
hypothesis of Lemma~\ref{lem:tree-unicyclic}(ii).  Applying that lemma to
\(F=f_N\) and using \eqref{eq:tree-positive-factor-bound}, we obtain
\begin{align}
 \left|\int
 \left(\prod_{e\in E(T)}f_e\right)f_{e'}
 \prod_{\substack{e\in R(T)\\e<e'}}(1+f_e)
 \mu^{\otimes S}(\dd X_S)\right|
 &\le C^n\int
 \left(\prod_{e\in E(T)}\abs{f_e}\right)\abs{f_{e'}}
 \mu^{\otimes S}(\dd X_S)\notag\\
 &\le C^n\norm{f_N}_{L^2(\mu^{\otimes2})}^2\ell_N^{n-2}.
 \label{eq:extra-edge-bound}
\end{align}
For each fixed \(T\), there are at most \(n^2/2\) terms in
\eqref{eq:tree-remainder-expansion}, and there are
\(n^{n-2}\) trees on \(S\) by Cayley's formula.  Combining
\eqref{eq:zeta-two-term-split}, \eqref{eq:tree-integral-bound},
\eqref{eq:extra-edge-bound}, and the estimates
\eqref{eq:ellN-bound}, \eqref{eq:fN-mean-bound}, and
\eqref{eq:fN-L2-bound}, we obtain
\begin{equation}\label{eq:zeta-fixed-set-bound}
 \abs{\zeta_N(S)}
 \le \frac{C^n n^n}{N^n}
 a^2\norm{K}_{L^2(\mu^{\otimes2})}^2
 \left(a\norm{K_-}_{L^\infty(\mu^{\otimes2})}\right)^{n-2}.
\end{equation}
When \(n=2\), the second sum in \eqref{eq:zeta-two-term-split} is zero,
and the same estimate follows from \eqref{eq:tree-integral-bound}.

There are \(\binom{N}{n}\) choices of \(S\).  Since
\(\binom{N}{n}\le N^n/n!\) and \(n!\ge(n/e)^n\),
we deduce from \eqref{eq:zeta-fixed-set-bound} that
\begin{equation}\label{eq:zeta-size-bound}
 \sum_{\abs S=n}\abs{\zeta_N(S)}
 \le Ca^2\norm{K}_{L^2(\mu^{\otimes2})}^2
 \left(Ca\norm{K_-}_{L^\infty(\mu^{\otimes2})}\right)^{n-2}.
\end{equation}
Choose \(c>0\) sufficiently small so that \(Cc\le1/2\).  We then deduce from
\eqref{eq:zeta-size-bound} that
\[
 \sum_{\abs S=n}\abs{\zeta_N(S)}
 \le Ca^2\norm{K}_{L^2(\mu^{\otimes2})}^2 2^{-(n-2)},
 \qquad 2\le n\le N.
\]
Summing over \(n\) proves \eqref{eq:zeta-target}.  Substitution into
\eqref{eq:connected-upper-bound} proves the upper bound in
\eqref{eq:canonical-bound}.
\end{proof}

\subsection{Heat kernel decomposition}
\label{subsec:heat-decomposition}

Let \(p_t(x)=(4\pi t)^{-d/2}e^{-\abs{x}^2/(4t)}\) be the heat kernel.  Fix
\(g=g_{d,s}^{(m)}\), where \(0\le s<d/2\), \(m\ge0\), and \(m=0\) if
\(s=0\).  For \(0<\tau\le1\), we split its heat representation at
\(t=\tau\) by setting
\begin{equation}\label{eq:heat-decomposition}
 r_\tau(x)
 :=\frac1{\Gamma((d-s)/2)}
 \int_0^\tau t^{(d-s)/2-1}e^{-m^2t}p_t(x)\dd t,
 \qquad
 g_\tau:=g-r_\tau.
\end{equation}
Here \(\Gamma\) denotes the gamma function.
Thus \(g=g_\tau+r_\tau\): \(r_\tau\) is the short-range part, retaining the
singularity of \(g\) but decaying rapidly beyond the scale \(\sqrt{\tau}\),
while \(g_\tau\) is regular at the origin.
We first establish the properties of the decomposition
\eqref{eq:heat-decomposition} in the two cases.

\medskip
\noindent\textit{The logarithmic case \(s=m=0\).}\par
\medskip

In this case, the decomposition \eqref{eq:heat-decomposition} takes the form
\begin{equation}\label{eq:log-split}
 r_\tau(x)=\frac1{\Gamma(d/2)}
 \int_0^\tau t^{d/2-1}p_t(x)\dd t,
 \qquad
 g_\tau=g_{d,0}-r_\tau.
\end{equation}

\begin{lemma}
\label{lem:log-coarse-positive}
Assume that \(s=m=0\).  Let \(\tau>0\), and let \(\nu\) be a finite signed
measure satisfying
\[
 \nu(\mathbb R^d)=0,
 \qquad
 \int\log(2+\abs{x})\,\abs{\nu}(\dd x)<\infty.
\]
Then
\begin{equation}\label{eq:log-coarse-heat}
 \iint g_\tau(x-y)\,\nu(\dd x)\nu(\dd y)
 =\frac1{\Gamma(d/2)}\int_\tau^\infty
 t^{d/2-1}\norm{p_{t/2}*\nu}_2^2\dd t.
\end{equation}
Consequently, for the density \(\rho\) in
Assumption~\ref{ass:kernel-density}, \((g_\tau)_\rho^\circ\) is positive
semidefinite.
\end{lemma}

\begin{proof}
By the normalization of \(g_{d,0}\) and the definition of \(r_\tau\) in
\eqref{eq:log-split}, for \(\xi\ne0\),
\[
 \widehat{r_\tau}(\xi)
 =\frac1{\Gamma(d/2)}
   \int_0^\tau t^{d/2-1}e^{-t\abs{\xi}^2}\dd t,
 \qquad
 \widehat{g_\tau}(\xi)
 =\frac1{\Gamma(d/2)}
   \int_\tau^\infty t^{d/2-1}e^{-t\abs{\xi}^2}\dd t.
\]
The assumptions on \(\nu\) imply that
\(\widehat{g_\tau}\abs{\widehat\nu}^2\in L^1(\mathbb R^d)\).

Since \(\nu(\mathbb R^d)=0\), adding a constant to \(g_\tau\) does not change
the integral.  By applying Plancherel's theorem and Fubini's theorem first to
\(p_\varepsilon*\nu\) and then letting \(\varepsilon\downarrow0\), we obtain
\[
 \begin{aligned}
 \iint g_\tau(x-y)\,\nu(\dd x)\nu(\dd y)
 &=(2\pi)^{-d}\int_{\mathbb R^d}
   \widehat{g_\tau}(\xi)\abs{\widehat\nu(\xi)}^2\dd\xi\\
 &=\frac{(2\pi)^{-d}}{\Gamma(d/2)}
   \int_{\mathbb R^d}\int_\tau^\infty
   t^{d/2-1}e^{-t\abs{\xi}^2}\abs{\widehat\nu(\xi)}^2
   \dd t\dd\xi\\
 &=\frac1{\Gamma(d/2)}\int_\tau^\infty
   t^{d/2-1}\norm{p_{t/2}*\nu}_2^2\dd t\ge0.
 \end{aligned}
\]
This proves \eqref{eq:log-coarse-heat}.

For \(a_1,\ldots,a_\ell\in\mathbb R\) and
\(y_1,\ldots,y_\ell\in\mathbb R^d\), set
\(\nu=\sum_i a_i\delta_{y_i}-(\sum_i a_i)\rho(x)\dd x\).
By assumption \eqref{eq:log-tail-assumption}, this measure satisfies the
assumptions in the statement, and
\[
 \sum_{i,j=1}^{\ell}a_ia_j(g_\tau)_\rho^\circ(y_i,y_j)
 =\iint g_\tau(x-y)\,\nu(\dd x)\nu(\dd y)\ge0.
\]
Thus \((g_\tau)_\rho^\circ\) is positive semidefinite.
\end{proof}

For the density \(\rho\) in Assumption~\ref{ass:kernel-density}, set
\[
 d_\tau(x):=(g_\tau)_\rho^\circ(x,x).
\]
By Lemma~\ref{lem:log-coarse-positive}, \((g_\tau)_\rho^\circ\) is positive
semidefinite, so \(d_\tau\ge0\).  Since \(g_{d,0}=g_\tau+r_\tau\), we have
\begin{equation}\label{eq:log-diagonal-expansion}
 d_\tau(x)
 =g_\tau(0)-2(g_{d,0}*\rho)(x)+2(r_\tau*\rho)(x)
  +\int(g_{d,0}*\rho)\rho-\int(r_\tau*\rho)\rho.
\end{equation}
We estimate the five terms in \eqref{eq:log-diagonal-expansion}.

For the first term, comparing the cutoffs \(\tau\) and \(1\) in
\eqref{eq:log-split}, we obtain
\begin{equation}\label{eq:log-coarse-origin}
 g_\tau(0)
 \le C_d+\frac1{\Gamma(d/2)}
   \int_\tau^1 t^{d/2-1}p_t(0)\dd t
 =C_d+\frac{(4\pi)^{-d/2}}{\Gamma(d/2)}\log(1/\tau)
 \le C_d\left(1+\log(1/\tau)\right).
\end{equation}
For the second and fourth terms, by the boundedness of \(\rho\), the
logarithmic moment condition, and
\(\abs{x-y}\le(1+\abs{x})(1+\abs{y})\), we obtain
\begin{equation}\label{eq:log-potential-bounds}
 -2(g_{d,0}*\rho)(x)
 \le C_{d,\rho}\left(1+\log(1+\abs{x})\right),
 \qquad
 \left|\int(g_{d,0}*\rho)\rho\right|\le C_{d,\rho}.
\end{equation}
For the third and fifth terms, by \eqref{eq:log-split}, \(\int p_t=1\), and
the scaling \(r_\tau(\sqrt{\tau}z)=r_1(z)\), we obtain
\begin{equation}\label{eq:log-remainder-scales}
 \norm{r_\tau}_1+\norm{r_\tau}_2^2
 =\tau^{d/2}\left(\norm{r_1}_1+\norm{r_1}_2^2\right)
 \le C_d\tau^{d/2}.
\end{equation}
Furthermore, since \(r_\tau\) is nonnegative and \(\rho\) is bounded,
\begin{equation}\label{eq:log-remainder-bounds}
 0\le2(r_\tau*\rho)(x)\le C_{d,\rho}\tau^{d/2},
 \qquad
 -\int(r_\tau*\rho)\rho\le0.
\end{equation}
Combining \eqref{eq:log-diagonal-expansion},
\eqref{eq:log-coarse-origin}, \eqref{eq:log-potential-bounds}, and
\eqref{eq:log-remainder-bounds}, we obtain, for every \(x\in\mathbb R^d\),
\begin{equation}\label{eq:log-diagonal-pointwise}
 0\le d_\tau(x)
 \le C_{d,\rho}
 \left(1+\log(1/\tau)+\log(1+\abs{x})\right).
\end{equation}

To control the contribution of \(g_\tau\) to the partition function, we need
a bound on the exponential moment of \(d_\tau\), uniformly in \(N\).  The
tail assumption \eqref{eq:log-tail-assumption} implies
\begin{equation}\label{eq:log-polynomial-moment}
 \log\int_{\mathbb R^d}(1+\abs{x})^\lambda\rho(x)\dd x
 \le C_{d,\rho}\lambda\log(2+\lambda),
 \qquad \lambda\ge0.
\end{equation}

By \eqref{eq:log-diagonal-pointwise},
\begin{equation}\label{eq:log-diagonal-exponential-reduction}
 \int_{\mathbb R^d}e^{\beta d_\tau(x)/(N-1)}\rho(x)\dd x
 \le
 \exp\left(\frac{C_{d,\rho}\beta}{N-1}
 \bigl(1+\log(1/\tau)\bigr)\right)
 \int_{\mathbb R^d}
 (1+\abs{x})^{C_{d,\rho}\beta/(N-1)}\rho(x)\dd x.
\end{equation}
Applying \eqref{eq:log-polynomial-moment} with
\(\lambda=C_{d,\rho}\beta/(N-1)\) to
\eqref{eq:log-diagonal-exponential-reduction}, and using
\(N/(N-1)\le2\), we obtain
\begin{equation}\label{eq:log-diagonal-cost}
 \sup_{N\ge2}\frac N2
 \log\int_{\mathbb R^d}
 \exp\left(\frac{\beta}{N-1}d_\tau(x)\right)\rho(x)\dd x
 \le C_{d,\rho}\beta
 \left(1+\log(1/\tau)+\log(2+\beta)\right).
\end{equation}
The constant \(C_{d,\rho}\) in \eqref{eq:log-diagonal-cost} depends only on
\(d\), \(\norm{\rho}_\infty\), and the parameters and moment bound in
\eqref{eq:log-tail-assumption}.

\medskip
\noindent\textit{The case \(0<s<d/2\) and \(m\ge0\).}\par
\medskip

By the gamma integral formula, the Riesz and Bessel--Riesz kernels have the
heat representation
\begin{equation}\label{eq:riesz-split}
 g_{d,s}^{(m)}
 =\frac1{\Gamma((d-s)/2)}\int_0^\infty
 t^{(d-s)/2-1}e^{-m^2t}p_t\dd t.
\end{equation}
Indeed, the Fourier multiplier of the right-hand side is
\[
 \frac1{\Gamma((d-s)/2)}\int_0^\infty
 t^{(d-s)/2-1}e^{-t(m^2+\abs{\xi}^2)}\dd t
 =(m^2+\abs{\xi}^2)^{-(d-s)/2}.
\]
Thus \(r_\tau\) and \(g_\tau\)
in the heat kernel decomposition
\eqref{eq:heat-decomposition} are,
respectively, the integrals over \((0,\tau)\) and \((\tau,\infty)\).

We first note that
\begin{equation}\label{eq:riesz-remainder-l1}
 r_\tau\ge0,
 \qquad
 \norm{r_\tau}_1
 \le\frac1{\Gamma((d-s)/2)}
 \int_0^\tau t^{(d-s)/2-1}\dd t
 \le C_{d,s}\tau^{(d-s)/2}.
\end{equation}
By Minkowski's inequality and \(\norm{p_t}_2=C_dt^{-d/4}\), we obtain
\begin{equation}\label{eq:riesz-remainder-l2}
 \norm{r_\tau}_2^2
 \le C_{d,s}\left(
 \int_0^\tau t^{(d-2s)/4-1}\dd t
 \right)^2
 \le C_{d,s}\tau^{d/2-s}.
\end{equation}
We deduce from \eqref{eq:riesz-split} that
\begin{equation}\label{eq:riesz-coarse-origin}
 \begin{aligned}
 g_\tau(0)
 &=\frac{(4\pi)^{-d/2}}{\Gamma((d-s)/2)}
 \int_\tau^\infty t^{-s/2-1}e^{-m^2t}\dd t\\
 &\le \frac{(4\pi)^{-d/2}}{\Gamma((d-s)/2)}
 \int_\tau^\infty t^{-s/2-1}\dd t
 =\frac{2(4\pi)^{-d/2}}{s\Gamma((d-s)/2)}\tau^{-s/2}
 \le C_{d,s}\tau^{-s/2}.
 \end{aligned}
\end{equation}

\begin{lemma}
For every finite signed measure \(\nu\) and every \(\tau\in(0,1]\),
\begin{equation}\label{eq:riesz-coarse-positive}
 \iint g_\tau(x-y)\,\nu(\dd x)\nu(\dd y)
 =\frac1{\Gamma((d-s)/2)}\int_\tau^\infty
 t^{(d-s)/2-1}e^{-m^2t}\norm{p_{t/2}*\nu}_2^2\dd t
 \ge0.
\end{equation}
Consequently, \((g_\tau)_\rho^\circ\) is positive semidefinite, and its
diagonal \(d_\tau(x):=(g_\tau)_\rho^\circ(x,x)\) satisfies
\begin{equation}\label{eq:riesz-coarse-diagonal}
 0\le d_\tau(x)\le2g_\tau(0)\le C_{d,s}\tau^{-s/2},
 \qquad x\in\mathbb R^d.
\end{equation}
\end{lemma}

\begin{proof}
Using the same argument as in Lemma~\ref{lem:log-coarse-positive}, we obtain
\eqref{eq:riesz-coarse-positive} and conclude that
\((g_\tau)_\rho^\circ\) is positive semidefinite.  When \(s>0\), no
condition on \(\nu(\mathbb R^d)\) is needed.

Using \(0\le g_\tau(z)\le g_\tau(0)\) and
\eqref{eq:riesz-coarse-origin}, we obtain
\[
 0\le d_\tau(x)
 \le g_\tau(0)+\int(g_\tau*\rho)\rho
 \le2g_\tau(0)
 \le C_{d,s}\tau^{-s/2},
\]
where the penultimate inequality follows from
\((g_\tau*\rho)(x)\le g_\tau(0)\) and \(\int_{\mathbb R^d}\rho=1\).
\end{proof}

We next derive an upper bound for the partition function
associated with the full kernel \(g=g_\tau+r_\tau\).

\begin{proposition}
\label{prop:coarse-fine-estimate}
Under Assumption~\ref{ass:kernel-density}, there exist universal constants
\(c,C>0\) such that, for every \(\beta>0\) satisfying
\[
 \beta\norm{\rho}_\infty\norm{r_\tau}_1\le c,
\]
and every \(N\ge2\),
\[
 \log\cZ_{N,\rho}(\beta;g)
 \le\frac N2\log\int_{\mathbb R^d}
 \exp\left(\frac{\beta}{N-1}d_\tau(x)\right)\rho(x)\dd x
 +C\beta^2\norm{\rho}_\infty\norm{r_\tau}_2^2.
\]
\end{proposition}

\begin{proof}
By the Cauchy--Schwarz inequality and
\(\cU_{N,\rho}(g)=\cU_{N,\rho}(g_\tau)+\cU_{N,\rho}(r_\tau)\), we obtain
\begin{equation}\label{eq:coarse-fine-cauchy-schwarz}
 \cZ_{N,\rho}(\beta;g)
 \le
 \left(\E e^{-2\beta\cU_{N,\rho}(g_\tau)}\right)^{1/2}
 \left(\E e^{-2\beta\cU_{N,\rho}(r_\tau)}\right)^{1/2}.
\end{equation}

We first estimate the contribution of \(g_\tau\).  By
Lemma~\ref{lem:log-coarse-positive} in the logarithmic case and
\eqref{eq:riesz-coarse-positive} when \(s>0\),
\((g_\tau)_\rho^\circ\) is positive semidefinite.  Hence, for every
configuration,
\[
 \cU_{N,\rho}(g_\tau;X^{1:N})
 =\frac1{2(N-1)}
 \left[\sum_{i,j=1}^N(g_\tau)_\rho^\circ(X_i,X_j)
 -\sum_{i=1}^Nd_\tau(X_i)\right]
 \ge-\frac1{2(N-1)}\sum_{i=1}^Nd_\tau(X_i).
\]
Consequently,
\begin{equation}\label{eq:coarse-part-exponential}
 \E e^{-2\beta\cU_{N,\rho}(g_\tau)}
 \le
 \left(\int_{\mathbb R^d}
 e^{\beta d_\tau(x)/(N-1)}\rho(x)\dd x\right)^N.
\end{equation}

We next estimate the contribution of \(r_\tau\).  The kernel
\((r_\tau)_\rho^\circ\) satisfies the two cancellations in
\eqref{eq:canonical-assumptions}, and
\[
 \cU_{N,\rho}(r_\tau;X^{1:N})
 =\frac1{N-1}\sum_{i<j}(r_\tau)_\rho^\circ(X_i,X_j).
\]
By the definition of \(r_\tau\) in the heat kernel
decomposition \eqref{eq:heat-decomposition}, \(r_\tau\ge0\), and hence
\[
 (r_\tau)_\rho^\circ(x,y)
 \ge-(r_\tau*\rho)(x)-(r_\tau*\rho)(y),
 \qquad
 \norm{\bigl((r_\tau)_\rho^\circ\bigr)_-}_{L^\infty(\rho^{\otimes2})}
 \le2\norm{\rho}_\infty\norm{r_\tau}_1.
\]
By the definition of \((r_\tau)_\rho^\circ\) and the two cancellations in
\eqref{eq:canonical-assumptions}, we have
\[
 \norm{(r_\tau)_\rho^\circ}_{L^2(\rho^{\otimes2})}^2
 =\iint (r_\tau)_\rho^\circ(x,y)r_\tau(x-y)
 \rho(x)\rho(y)\dd x\dd y.
\]
By the Cauchy--Schwarz inequality, we obtain
\[
 \norm{(r_\tau)_\rho^\circ}_{L^2(\rho^{\otimes2})}^2
 \le\iint r_\tau(x-y)^2\rho(x)\rho(y)\dd x\dd y
 \le\norm{\rho}_\infty\norm{r_\tau}_2^2.
\]
Under the condition \(\beta\norm{\rho}_\infty\norm{r_\tau}_1\le c\), the
bounds for the negative part and the \(L^2\) norm verify the assumptions of
Theorem~\ref{thm:canonical} for \((r_\tau)_\rho^\circ\) with \(a=2\beta\).
Therefore,
\begin{equation}\label{eq:short-part-exponential}
 \E e^{-2\beta\cU_{N,\rho}(r_\tau)}
 \le \exp\left(C\beta^2\norm{\rho}_\infty\norm{r_\tau}_2^2\right).
\end{equation}
Combining \eqref{eq:coarse-fine-cauchy-schwarz},
\eqref{eq:coarse-part-exponential}, and
\eqref{eq:short-part-exponential}, we obtain
\[
 \cZ_{N,\rho}(\beta;g)
 \le
 \left(\int_{\mathbb R^d}
 e^{\beta d_\tau(x)/(N-1)}\rho(x)\dd x\right)^{N/2}
 \exp\left(C\beta^2\norm{\rho}_\infty\norm{r_\tau}_2^2\right).
\]
Taking logarithms completes the proof.
\end{proof}

\subsection{Proof of uniform partition-function estimates}
\label{subsec:proof-uniform-partition}

\begin{proof}[Proof of Theorem~\ref{thm:uniform-partition}]
Set \(g=g_{d,s}^{(m)}\).
Since
\(\E_{\rho^{\otimes N}}\cU_{N,\rho}(g)=0\), Jensen's inequality gives
\[
 \cZ_{N,\rho}(\beta;g)\ge1,
 \qquad \beta\ge0.
\]
It remains to prove the upper bounds.  For each kernel, we first choose a
\(\beta\)-dependent cutoff in \eqref{eq:heat-decomposition} and apply
Proposition~\ref{prop:coarse-fine-estimate} for every \(\beta>0\).  The
resulting estimate is uniform in \(N\) and has the stated large-\(\beta\) order.
We then improve the bound near \(\beta=0\) to quadratic order.

\medskip
\noindent\textit{The logarithmic kernel.}\par
\medskip

Assume that \(s=m=0\).  For \(\varepsilon\in(0,1]\) and \(\beta>0\), set
\[
 \tau=\left(\frac{\varepsilon}{1+\beta}\right)^{2/d}.
\]
By the logarithmic remainder estimate
\eqref{eq:log-remainder-scales},
\[
 \beta\norm{\rho}_\infty\norm{r_\tau}_1
 \le C_{d,\rho}\beta\tau^{d/2}
 =C_{d,\rho}\varepsilon\frac{\beta}{1+\beta}
 \le C_{d,\rho}\varepsilon.
\]
Choose \(\varepsilon\in(0,1]\), depending only on \(d\) and \(\rho\), so that
\(C_{d,\rho}\varepsilon\le c_0\), where \(c_0\) denotes the constant \(c\) in
Proposition~\ref{prop:coarse-fine-estimate}.  Then the smallness condition in
Proposition~\ref{prop:coarse-fine-estimate} is satisfied.  By
Proposition~\ref{prop:coarse-fine-estimate},
\eqref{eq:log-diagonal-cost}, \eqref{eq:log-remainder-scales}, and the choice
of \(\tau\), we obtain
\begin{equation}\label{eq:log-all-beta-bound}
\begin{aligned}
 \sup_{N\ge2}\log\cZ_{N,\rho}(\beta;g_{d,0})
 &\le C_{d,\rho}\beta
 \left(1+\log((1+\beta)/\varepsilon)+\log(2+\beta)\right)
 +C_{d,\rho}\frac{\beta^2}{1+\beta}\\
 &\le C_{d,\rho}\beta\log(2+\beta),
 \qquad \beta>0.
\end{aligned}
\end{equation}
For \(\beta\ge1\), \eqref{eq:log-all-beta-bound} is the \(s=0\) estimate in
\eqref{eq:riesz-parameter-bound}.  Taking \(\beta=2\) in
\eqref{eq:log-all-beta-bound}, we obtain
\begin{equation}\label{eq:log-fixed-exponential-moment}
 \sup_{N\ge2}\E_{\rho^{\otimes N}}
 e^{-2\cU_{N,\rho}(g_{d,0})}<\infty.
\end{equation}

To obtain the quadratic bound for \(0\le\beta\le1\), we first estimate the
second moment.
Notice that, by the two cancellations in \eqref{eq:canonical-centering},
\[
 \E_{\rho^{\otimes N}}\!\left[
 (g_{d,0})_\rho^\circ(X_i,X_j)
 (g_{d,0})_\rho^\circ(X_k,X_\ell)
 \right]=0,
 \qquad \{i,j\}\ne\{k,\ell\}.
\]
Hence,
\begin{equation}\label{eq:log-second-moment}
\begin{aligned}
 \E_{\rho^{\otimes N}}\!\left[\abs{\cU_{N,\rho}(g_{d,0})}^2\right]
 &=\frac{N}{2(N-1)}
   \norm{(g_{d,0})_\rho^\circ}_{L^2(\rho^{\otimes2})}^2\\
 &\le\norm{(g_{d,0})_\rho^\circ}_{L^2(\rho^{\otimes2})}^2
 \le\norm{g_{d,0}(x-y)}_{L^2(\rho^{\otimes2})}^2<\infty.
\end{aligned}
\end{equation}
Here we used the \(L^2\)-contraction of
\(K\mapsto K_\rho^\circ\), the boundedness of
\(\rho\), the local square integrability of the logarithmic singularity, and
the assumption
\eqref{eq:log-tail-assumption}.

Taylor's formula with integral remainder gives, for every
\(y\in\mathbb R\) and \(0\le\beta\le1\),
\begin{equation}\label{eq:small-beta-taylor}
 0\le e^{-\beta y}-1+\beta y
 =\beta^2y^2\int_0^1(1-t)e^{-t\beta y}\dd t
 \le C\beta^2\left(y^2+e^{-2y}\right).
\end{equation}
Taking \(y=\cU_{N,\rho}(g_{d,0})\) in
\eqref{eq:small-beta-taylor}, and
using
\(\E_{\rho^{\otimes N}}\cU_{N,\rho}(g_{d,0})=0\), the second moment bound
\eqref{eq:log-second-moment}, and the fixed exponential moment bound
\eqref{eq:log-fixed-exponential-moment}, we obtain
\[
 1\le\cZ_{N,\rho}(\beta;g_{d,0})
 \le1+C_{d,\rho}\beta^2,
 \qquad 0\le\beta\le1.
\]
Thus,
\[
 0\le\log\cZ_{N,\rho}(\beta;g_{d,0})
 \le C_{d,\rho}\beta^2,
 \qquad 0\le\beta\le1.
\]

\medskip
\noindent\textit{The Riesz and Bessel--Riesz kernels.}\par
\medskip

Assume that \(0<s<d/2\).  For \(\varepsilon\in(0,1]\) and \(\beta>0\), set
\[
 \tau=\left(\frac{\varepsilon}{1+\beta}\right)^{2/(d-s)}.
\]
By the estimate
\eqref{eq:riesz-remainder-l1},
\[
 \beta\norm{\rho}_\infty\norm{r_\tau}_1
 \le C_{d,s,\rho}\beta\tau^{(d-s)/2}
 =C_{d,s,\rho}\varepsilon\frac{\beta}{1+\beta}
 \le C_{d,s,\rho}\varepsilon.
\]
As in the logarithmic case, choose \(\varepsilon\in(0,1]\) sufficiently small.
By Proposition~\ref{prop:coarse-fine-estimate}, the
bounds \eqref{eq:riesz-remainder-l2} and
\eqref{eq:riesz-coarse-diagonal}, and the choice of \(\tau\), we obtain
\begin{equation}\label{eq:riesz-all-beta-bound}
\begin{aligned}
 \sup_{N\ge2}\log\cZ_{N,\rho}(\beta;g_{d,s}^{(m)})
 &\le C_{d,s,\rho}\beta\tau^{-s/2}
      +C_{d,s,\rho}\beta^2\tau^{d/2-s}\\
 &\le C_{d,s,m,\rho}\beta(1+\beta)^{s/(d-s)},
 \qquad \beta>0.
\end{aligned}
\end{equation}
For \(\beta\ge1\), by \eqref{eq:riesz-all-beta-bound} and
\(1+\beta\le2\beta\), we obtain the \(0<s<d/2\) estimate in
\eqref{eq:riesz-parameter-bound}.

For \(0\le\beta\le1\), by \eqref{eq:riesz-all-beta-bound} at \(\beta=2\),
\(0\le g\le g_{d,s}\), \(2s<d\), and the two cancellations in
\eqref{eq:canonical-centering}, we obtain
\[
 \sup_{N\ge2}\E_{\rho^{\otimes N}}e^{-2\cU_{N,\rho}(g)}<\infty,
 \qquad
 \sup_{N\ge2}\E_{\rho^{\otimes N}}\abs{\cU_{N,\rho}(g)}^2
 \le\norm{g_\rho^\circ}_{L^2(\rho^{\otimes2})}^2<\infty.
\]
Applying \eqref{eq:small-beta-taylor} with
\(y=\cU_{N,\rho}(g)\), and using
\eqref{eq:canonical-centering}, we obtain
\[
 0\le\log\cZ_{N,\rho}(\beta;g)
 \le C_{d,s,m,\rho}\beta^2,
 \qquad 0\le\beta\le1.
\]
This completes the proof.
\end{proof}

\section{Application I: Entropic mean-field convergence near thermal equilibrium}
\label{sec:gibbs-integrability}

In this section, we first prove uniform R\'enyi divergence and relative
entropy bounds for the centered Gibbs measure on position space.  At thermal
equilibrium, its density is exactly the Radon--Nikodym derivative of the
interacting Gibbs equilibrium with respect to the product equilibrium on
phase space.  We use this identity to obtain a time-uniform entropy estimate.

Let \(g=g_{d,s}^{(m)}\), and let \(\rho\) satisfy
Assumption~\ref{ass:kernel-density}.  Recall that
\(\cP_{N,\rho}^{\beta}\) is the centered Gibbs measure defined by
\eqref{eq:centered-gibbs-measure}.

\begin{proposition}
\label{prop:gibbs-entropy}
For every \(\beta>0\) and \(p>1\), we have
\[
 \sup_{N\ge2}\left\{
  D_p(\cP_{N,\rho}^{\beta}\mid\rho^{\otimes N})
  +\Ent(\cP_{N,\rho}^{\beta}\mid\rho^{\otimes N})
  +\Ent(\rho^{\otimes N}\mid\cP_{N,\rho}^{\beta})
 \right\}<\infty.
\]
Furthermore, for every \(\beta_N\ge1\),
\begin{equation}\label{eq:gibbs-betaN-entropy}
 \begin{aligned}
  \Ent(\cP_{N,\rho}^{\beta_N}\mid\rho^{\otimes N})
  &\le D_2(\cP_{N,\rho}^{\beta_N}\mid\rho^{\otimes N})\\
  &\le C_{d,s,m,\rho}
  \begin{cases}
   \beta_N\log(2+\beta_N),&s=0,\\
   \beta_N^{d/(d-s)},&0<s<d/2.
  \end{cases}
 \end{aligned}
\end{equation}
\end{proposition}

\begin{proof}
By a direct calculation,
\begin{align*}
 D_p(\cP_{N,\rho}^{\beta}\mid\rho^{\otimes N})
 &=\frac{\log\cZ_{N,\rho}(p\beta;g)
       -p\log\cZ_{N,\rho}(\beta;g)}{p-1},\\
 \Ent(\rho^{\otimes N}\mid\cP_{N,\rho}^{\beta})
 &=\log\cZ_{N,\rho}(\beta;g),
\end{align*}
where the second identity follows from
\(\E_{\rho^{\otimes N}}\cU_{N,\rho}(g)=0\).  Since
\(\cZ_{N,\rho}(a;g)\ge1\) by Jensen's inequality and relative entropy is
bounded by R\'enyi divergence,
\begin{equation}\label{eq:gibbs-renyi-partition}
 \Ent(\cP_{N,\rho}^{\beta}\mid\rho^{\otimes N})
 \le D_p(\cP_{N,\rho}^{\beta}\mid\rho^{\otimes N})
 \le \frac{\log\cZ_{N,\rho}(p\beta;g)}{p-1}.
\end{equation}
By Theorem~\ref{thm:uniform-partition} at \(\beta\) and \(p\beta\),
\(D_p(\cP_{N,\rho}^{\beta}\mid\rho^{\otimes N})\) and both relative
entropies are bounded uniformly in \(N\).

Taking \(\beta=\beta_N\) and \(p=2\) in the
estimate \eqref{eq:gibbs-renyi-partition}, and using a uniform
partition-function estimate with explicit parameter dependence, namely
\eqref{eq:riesz-parameter-bound}, at \(2\beta_N\), we obtain
\eqref{eq:gibbs-betaN-entropy}.
\end{proof}

We now turn to thermal equilibrium.  Fix \(\beta>0\) and
\(V\in C^1(\mathbb R^d)\).  Suppose that a probability density
\(\rho_\beta\) satisfying Assumption~\ref{ass:kernel-density} solves
\[
 \rho_\beta(x)=z_{\beta,V}^{-1}
 \exp\left(-\beta V(x)-\beta(g*\rho_\beta)(x)\right),
\]
where \(z_{\beta,V}>0\) is the normalizing constant.  Set
\[
 m_\beta(x,v):=\rho_\beta(x)
 \left(\frac{\beta}{2\pi}\right)^{d/2}e^{-\beta\abs{v}^2/2},
 \qquad
 \cM_{N,\beta}:=m_\beta^{\otimes N}.
\]
Writing \(X^{1:N}=(x_1,\ldots,x_N)\), \(Z_i=(x_i,v_i)\),
\(Z^{1:N}=(Z_1,\ldots,Z_N)\), and
\(\dd Z^{1:N}=\prod_{i=1}^N\dd x_i\dd v_i\), define
\[
 H_{N,V}(Z^{1:N})
 :=\frac12\sum_{i=1}^N\abs{v_i}^2+\sum_{i=1}^N V(x_i)
  +\frac1{2(N-1)}\sum_{i=1}^N\sum_{j\ne i}g(x_i-x_j)
\]
and
\[
 \cG_{N,\beta}(Z^{1:N})
 :=\frac{e^{-\beta H_{N,V}(Z^{1:N})}}
 {\displaystyle\int e^{-\beta H_{N,V}}\dd Z^{1:N}}.
\]
Thus \(\cM_{N,\beta}\) is the product equilibrium, whereas
\(\cG_{N,\beta}\) is the interacting \(N\)-particle Gibbs equilibrium.
By the equation for \(\rho_\beta\) and
the modulated-energy expansion
\eqref{eq:modulated-energy-expansion}, for some \(C>0\) independent of
\(Z^{1:N}\),
\[
 e^{-\beta H_{N,V}(Z^{1:N})}
 =C
  e^{-\beta\cU_{N,\rho_\beta}(g)}\cM_{N,\beta}(Z^{1:N}).
\]
By Theorem~\ref{thm:uniform-partition},
\(e^{-\beta\cU_{N,\rho_\beta}(g)}\) is integrable with respect to
\(\cM_{N,\beta}\).  Hence \(\cG_{N,\beta}\) is well defined, and after
normalization we obtain
\begin{equation}\label{eq:kinetic-gibbs-factorization}
 \frac{\cG_{N,\beta}}{\cM_{N,\beta}}(Z^{1:N})
 =\cZ_{N,\rho_\beta}(\beta;g)^{-1}
  e^{-\beta\cU_{N,\rho_\beta}(g;X^{1:N})}
 =\frac{\dd\cP_{N,\rho_\beta}^{\beta}}
        {\dd\rho_\beta^{\otimes N}}(X^{1:N}).
\end{equation}

Let \(F^N(t,Z^{1:N})\) satisfy the \(N\)-particle kinetic
Fokker--Planck equation
\begin{equation}\label{eq:N-particle-kinetic-fokker-planck}
 \begin{aligned}
  \partial_tF^N
  &+\sum_{i=1}^Nv_i\cdot\nabla_{x_i}F^N
  -\sum_{i=1}^N\left(
    \nabla V(x_i)+\frac1{N-1}\sum_{j\ne i}\nabla g(x_i-x_j)
   \right)\cdot\nabla_{v_i}F^N\\
  &=\sigma\sum_{i=1}^N\nabla_{v_i}\!\cdot\left(
    v_iF^N+\frac1\beta\nabla_{v_i}F^N
   \right).
 \end{aligned}
\end{equation}

For a probability density \(F\) on
\((\mathbb R^d\times\mathbb R^d)^N\), define the relative Fisher
information in the velocity variables by
\[
 \cI_v(F\mid\cG_{N,\beta})
 :=\sum_{i=1}^N\int F
 \left|\nabla_{v_i}\log\frac{F}{\cG_{N,\beta}}\right|^2\dd Z^{1:N}.
\]
For a smooth positive solution of
\eqref{eq:N-particle-kinetic-fokker-planck}, a direct calculation gives
\[
 \Ent(F^N(t)\mid\cG_{N,\beta})
 +\frac{\sigma}{\beta}\int_0^t
  \cI_v(F^N(r)\mid\cG_{N,\beta})\dd r
 =\Ent(F^N(0)\mid\cG_{N,\beta}).
\]
For weak solutions, an appropriate approximation and lower semicontinuity
yield the entropy inequality
\begin{equation}\label{eq:equilibrium-entropy-interface}
 \Ent(F^N(t)\mid\cG_{N,\beta})
 +\frac{\sigma}{\beta}\int_0^t
  \cI_v(F^N(r)\mid\cG_{N,\beta})\dd r
 \le\Ent(F^N(0)\mid\cG_{N,\beta}).
\end{equation}

\begin{proposition}
\label{prop:equilibrium-entropy-transfer}
Let \(F^N(t)\) be a weak solution of
\eqref{eq:N-particle-kinetic-fokker-planck} such that
\(\Ent(F^N(0)\mid\cG_{N,\beta})<\infty\).  For every \(p>1\), there exists
\(C<\infty\), independent of \(N\) and \(t\), such that, for all
\(N\ge2\) and \(t\ge0\),
\begin{equation}\label{eq:equilibrium-entropy-transfer}
 \begin{aligned}
  \Ent(F^N(t)\mid\cM_{N,\beta})
  &\le\frac{p}{p-1}\Ent(F^N(0)\mid\cG_{N,\beta})
     +\frac{1}{p-1}\log\cZ_{N,\rho_\beta}(p\beta;g)\\
  &\le\frac{p}{p-1}\Ent(F^N(0)\mid\cG_{N,\beta})+C.
 \end{aligned}
\end{equation}
If \(F^N(0)=\cM_{N,\beta}\), then
\[
 \Ent(F^N(t)\mid\cM_{N,\beta})
 \le\frac{1}{p-1}\log\cZ_{N,\rho_\beta}(p\beta;g)
 \le C.
\]
\end{proposition}

\begin{proof}
By the chain rule for relative entropy and the Donsker--Varadhan inequality
applied to
\((p-1)\log(\cG_{N,\beta}/\cM_{N,\beta})\),
\begin{equation}\label{eq:relative-entropy-renyi-transfer}
 \begin{aligned}
 \Ent(F^N(t)\mid\cM_{N,\beta})
 &=\Ent(F^N(t)\mid\cG_{N,\beta})
   +\int F^N(t)\log\frac{\cG_{N,\beta}}{\cM_{N,\beta}}\dd Z^{1:N}\\
 &\le\frac{p}{p-1}\Ent(F^N(t)\mid\cG_{N,\beta})
   +\frac{1}{p-1}
     \log\int\left(\frac{\cG_{N,\beta}}{\cM_{N,\beta}}\right)^{p-1}
          \cG_{N,\beta}\dd Z^{1:N}\\
 &=\frac{p}{p-1}\Ent(F^N(t)\mid\cG_{N,\beta})
   +D_p(\cG_{N,\beta}\mid\cM_{N,\beta}).
 \end{aligned}
\end{equation}
By the factorization
\eqref{eq:kinetic-gibbs-factorization} and the definition of \(D_p\),
\begin{equation}\label{eq:kinetic-renyi-partition}
 \begin{aligned}
 D_p(\cG_{N,\beta}\mid\cM_{N,\beta})
 &=D_p(\cP_{N,\rho_\beta}^{\beta}
        \mid\rho_\beta^{\otimes N})\\
 &=\frac{\log\cZ_{N,\rho_\beta}(p\beta;g)
       -p\log\cZ_{N,\rho_\beta}(\beta;g)}{p-1}\\
 &\le\frac{\log\cZ_{N,\rho_\beta}(p\beta;g)}{p-1},
 \end{aligned}
\end{equation}
where the last inequality follows from
\(\cZ_{N,\rho_\beta}(\beta;g)\ge1\).  Using
Theorem~\ref{thm:uniform-partition} at \(p\beta\),
\[
 \sup_{N\ge2}D_p(\cG_{N,\beta}\mid\cM_{N,\beta})
 =\sup_{N\ge2}D_p(\cP_{N,\rho_\beta}^{\beta}
        \mid\rho_\beta^{\otimes N})<\infty.
\]

Combining \eqref{eq:relative-entropy-renyi-transfer}
with the entropy-dissipation
inequality \eqref{eq:equilibrium-entropy-interface}, we obtain
\begin{align}
 &\Ent(F^N(t)\mid\cM_{N,\beta})
 +\frac{p}{p-1}\frac{\sigma}{\beta}\int_0^t
   \cI_v(F^N(r)\mid\cG_{N,\beta})\dd r\notag\\
 &\quad\le\frac{p}{p-1}\Ent(F^N(0)\mid\cG_{N,\beta})
 +D_p(\cG_{N,\beta}\mid\cM_{N,\beta}).
 \label{eq:entropy-transfer-with-fisher}
\end{align}
Dropping the Fisher information term in
\eqref{eq:entropy-transfer-with-fisher} and using
\eqref{eq:kinetic-renyi-partition}, we obtain
\eqref{eq:equilibrium-entropy-transfer}.

If \(F^N(0)=\cM_{N,\beta}\), then
\(\int\cU_{N,\rho_\beta}(g)\dd\cM_{N,\beta}=0\).
Hence, by the factorization
\eqref{eq:kinetic-gibbs-factorization},
\[
 \Ent(\cM_{N,\beta}\mid\cG_{N,\beta})
 =\log\cZ_{N,\rho_\beta}(\beta;g).
\]
Together with \eqref{eq:kinetic-renyi-partition}, this yields
\[
 \frac{p}{p-1}\Ent(\cM_{N,\beta}\mid\cG_{N,\beta})
 +D_p(\cG_{N,\beta}\mid\cM_{N,\beta})
 =\frac{1}{p-1}\log\cZ_{N,\rho_\beta}(p\beta;g).
\]
Substituting this identity into \eqref{eq:entropy-transfer-with-fisher}
proves the last assertion.
\end{proof}

\section{Application II: Commutator estimates and quantitative propagation of chaos}
\label{sec:optimal-modulated-free-energy}

In this section, we prove an \(N^{-1}\) bound for the normalized relative
entropy of the three-dimensional Coulomb system.  We first obtain a
commutator estimate from a uniform partition-function estimate and then
combine it with the evolution of the modulated free energy.  Although the dynamical
application below concerns the three-dimensional Coulomb kernel, the
commutator estimate applies to all logarithmic, Riesz, and Bessel--Riesz
kernels covered by Theorem~\ref{thm:uniform-partition} and, more generally,
to nonnegative repulsive potentials admitting a uniform partition-function
estimate and satisfying \(\abs{z}\,\abs{\nabla g(z)}\le C(g(z)+1)\).

\subsection{Commutator estimates}

Let \(g=g_{d,s}^{(m)}\), and let
\(u:\mathbb R^d\to\mathbb R^d\) be globally Lipschitz.  Away from the
diagonal, define
\begin{equation*}
 q_u(x,y):=(u(x)-u(y))\cdot\nabla g(x-y).
\end{equation*}
The kernel \(q_u\) is symmetric.

\begin{proposition}
\label{prop:gibbs-commutator-exponential}
Under Assumption~\ref{ass:kernel-density}, fix \(\beta>0\).  There exists
\(\theta>0\), depending only on \(d,s,m,\beta,\rho\), and
\(\norm{Du}_\infty\), such that
\begin{equation}\label{eq:gibbs-commutator-mgf}
 \sup_{N\ge2}\E_{\cP_{N,\rho}^{\beta}}
 e^{\theta|\cU_{N,\rho}(q_u)|}
 <\infty.
\end{equation}
Moreover, for some \(C<\infty\) independent of \(N\), every probability law
\(P_N\ll\cP_{N,\rho}^{\beta}\) satisfies
\begin{equation}\label{eq:average-coulomb-commutator}
 \frac1N\E_{P_N}|\cU_{N,\rho}(q_u)|
 \le \frac{C}{N}
 \left[\Ent(P_N\mid\cP_{N,\rho}^{\beta})+1\right].
\end{equation}
\end{proposition}

\begin{proof}
By the definition of the centered Gibbs measure
\(\cP_{N,\rho}^{\beta}\) in \eqref{eq:centered-gibbs-measure}, for either
sign and every \(\theta>0\),
\begin{equation*}
 \E_{\cP_{N,\rho}^{\beta}}
 e^{\mathbin{\pm}\theta\cU_{N,\rho}(q_u)}
 =\frac1{\cZ_{N,\rho}(\beta;g)}
  \E_{\rho^{\otimes N}}
  e^{-\beta\cU_{N,\rho}(g)
     \mathbin{\pm}\theta\cU_{N,\rho}(q_u)}.
\end{equation*}

Assume first that \(0<s<d/2\).  For \(m=0\), homogeneity gives
\(\abs{z}\,\abs{\nabla g_{d,s}(z)}=s g_{d,s}(z)\).  For \(m>0\), differentiating
the heat representation \eqref{eq:riesz-split} and using
\(a e^{-a}\le C e^{-a/2}\) gives
\[
 \abs{z}\,\abs{\nabla g_{d,s}^{(m)}(z)}
 \le C_{d,s,m}g_{d,s}^{(m)}(z/\sqrt2)
 \le C_{d,s,m}\bigl(g_{d,s}^{(m)}(z)+1\bigr).
\]
The last inequality follows from the \(|z|^{-s}\) behavior near the origin
and boundedness away from it.  Thus, in both cases,
\begin{equation}\label{eq:riesz-commutator-comparison}
 |q_u(x,y)|
 \le C_{d,s,m}\norm{Du}_\infty\bigl(g(x-y)+1\bigr).
\end{equation}
Whenever \(\beta',\theta>0\) satisfy
\(C_{d,s,m}\theta\norm{Du}_\infty\le\beta'/2\),
\eqref{eq:riesz-commutator-comparison} and \(g\ge0\) give
\[
 \begin{aligned}
  \beta'g\pm\theta q_u
  &\ge\beta'g-\theta|q_u|\\
  &\ge\bigl(\beta'-C_{d,s,m}\theta\norm{Du}_\infty\bigr)g
      -C_{d,s,m}\theta\norm{Du}_\infty\\
  &\ge-C_{d,s,m}\theta\norm{Du}_\infty,
 \end{aligned}
\]
and hence
\[
 (\beta'g\pm\theta q_u)_-
 \le C_{d,s,m}\theta\norm{Du}_\infty,
 \qquad
 |\beta' g\pm\theta q_u|
 \le C(\beta'+\theta\norm{Du}_\infty)(g+1).
\]
It follows that
\[
 \left\|\bigl((\beta' g\pm\theta q_u)_\rho^\circ\bigr)_-\right\|_\infty
 +\norm{(\beta' g\pm\theta q_u)_\rho^\circ}_{L^2(\rho^{\otimes2})}
 \le C(\beta'+\theta\norm{Du}_\infty),
\]
where we used the fact that
\(\norm{g*\rho}_\infty+\norm{g(x-y)}_{L^2(\rho^{\otimes2})}<\infty\);
this follows from the heat kernel estimates
\eqref{eq:riesz-remainder-l1}--\eqref{eq:riesz-coarse-origin}
at \(\tau=1\).
Choose \(0<\beta'\le\beta\) sufficiently small and then \(\theta>0\) so that
\[
 C_{d,s,m}\theta\norm{Du}_\infty\le\frac{\beta'}2,
 \qquad
 2C(\beta'+\theta\norm{Du}_\infty)\le c,
\]
where \(c\) is the constant in the smallness condition
\eqref{eq:canonical-smallness}.
Thus, \eqref{eq:canonical-smallness} holds with
\(K=(\beta'g\pm\theta q_u)_\rho^\circ\) and \(a=2\).
By Theorem~\ref{thm:canonical}, we have
\begin{equation}\label{eq:buffered-commutator-bound}
 \sup_{N\ge2}\E_{\rho^{\otimes N}}
 e^{-2\cU_{N,\rho}(\beta'g\pm\theta q_u)}<\infty.
\end{equation}
By linearity of \(K\mapsto K_\rho^\circ\), we find that
\[
 -\beta\cU_{N,\rho}(g)
 \mathbin{\pm}\theta\cU_{N,\rho}(q_u)
 =-(\beta-\beta')\cU_{N,\rho}(g)
 -\cU_{N,\rho}(\beta'g\mp\theta q_u).
\]
Since \(\beta'\le\beta\), we deduce from the Cauchy--Schwarz inequality,
Theorem~\ref{thm:uniform-partition}, and
the exponential bound
\eqref{eq:buffered-commutator-bound} that
\begin{equation}\label{eq:riesz-commutator-product-bound}
 \begin{aligned}
 &\sup_{N\ge2}\E_{\rho^{\otimes N}}
 e^{-\beta\cU_{N,\rho}(g)
    \mathbin{\pm}\theta\cU_{N,\rho}(q_u)}\\
 &\quad\le
 \left(\sup_{N\ge2}
 \cZ_{N,\rho}\bigl(2(\beta-\beta');g\bigr)\right)^{1/2}
 \left(\sup_{N\ge2}\E_{\rho^{\otimes N}}
 e^{-2\cU_{N,\rho}(\beta'g\mp\theta q_u)}\right)^{1/2}
 <\infty.
 \end{aligned}
\end{equation}

Suppose now that \(s=0\).  In this case,
\[
 |q_u(x,y)|\le C_d\norm{Du}_\infty.
\]
Hence \((q_u)_\rho^\circ\) is bounded.  Choose \(\theta>0\) sufficiently
small that
the condition \eqref{eq:canonical-smallness} holds for both
\(K=\pm(q_u)_\rho^\circ\) with \(a=2\theta\).
By Theorem~\ref{thm:canonical}, we have
\begin{equation*}
 \sup_{N\ge2}\E_{\rho^{\otimes N}}
 e^{\mathbin{\pm}2\theta\cU_{N,\rho}(q_u)}<\infty.
\end{equation*}
Using the same Cauchy--Schwarz argument as in
\eqref{eq:riesz-commutator-product-bound}, we obtain
\begin{equation}\label{eq:log-commutator-product-bound}
 \sup_{N\ge2}\E_{\rho^{\otimes N}}
 e^{-\beta\cU_{N,\rho}(g_{d,0})
    \mathbin{\pm}\theta\cU_{N,\rho}(q_u)}<\infty.
\end{equation}

Using \(e^{\theta|r|}\le e^{\theta r}+e^{-\theta r}\), we deduce
\eqref{eq:gibbs-commutator-mgf} from
\eqref{eq:riesz-commutator-product-bound} when \(0<s<d/2\) and from
\eqref{eq:log-commutator-product-bound} when \(s=0\).

Applying the Donsker--Varadhan inequality to
\(\theta|\cU_{N,\rho}(q_u)|\), we obtain
\begin{equation}\label{eq:commutator-donsker-varadhan}
 \theta\E_{P_N}|\cU_{N,\rho}(q_u)|
 \le\Ent(P_N\mid\cP_{N,\rho}^{\beta})
 +\log\E_{\cP_{N,\rho}^{\beta}}
 e^{\theta|\cU_{N,\rho}(q_u)|}.
\end{equation}
By the exponential moment bound
\eqref{eq:gibbs-commutator-mgf}, the logarithmic term in the estimate
\eqref{eq:commutator-donsker-varadhan} is bounded
uniformly in \(N\).
Dividing the estimate
\eqref{eq:commutator-donsker-varadhan} by \(\theta N\), we obtain
\eqref{eq:average-coulomb-commutator}.
\end{proof}

\subsection{Quantitative propagation of chaos}

We study the first-order Coulomb system on \(\mathbb R^3\),
\begin{equation*}
 \dd X_t^{N,i}
 =-\frac1{N-1}\sum_{j\ne i}\nabla g(X_t^{N,i}-X_t^{N,j})\dd t
   +\sqrt{\frac2\beta}\dd B_t^i,
 \qquad 1\le i\le N.
\end{equation*}
Here \(g=g_{3,1}=(4\pi|x|)^{-1}\), and
\(B^1,\ldots,B^N\) are independent standard Brownian motions in
\(\mathbb R^3\).
The associated \(N\)-particle Fokker--Planck equation is
\begin{equation}\label{eq:first-order-coulomb-fokker-planck}
 \partial_tF^N
 =\frac1\beta\sum_{i=1}^N\Delta_{x_i}F^N
 +\frac1{N-1}\sum_{i\ne j}\nabla_{x_i}\!\cdot
 \bigl(F^N\nabla g(x_i-x_j)\bigr).
\end{equation}
Formally, as \(N\to\infty\), the empirical measure converges to a solution of
the nonlinear Fokker--Planck equation
\begin{equation}\label{eq:first-order-coulomb-pde}
 \partial_t\rho_t
 -\nabla\!\cdot\bigl(\rho_t\nabla(g*\rho_t)\bigr)
 =\frac1\beta\Delta\rho_t.
\end{equation}

Set
\[
 G_N(X^{1:N})
 :=\exp\left[-\frac{\beta}{2(N-1)}
       \sum_{i\ne j}g(x_i-x_j)\right].
\]

We use the following notion of entropy solution.

\begin{definition}[Entropy solution]
\label{def:coulomb-entropy-solution}
For \(T>0\), a symmetric probability density
\(F^N\in L^\infty(0,T;L^1(\mathbb R^{3N}))\) is an entropy solution on \([0,T]\)
if it solves \eqref{eq:first-order-coulomb-fokker-planck} in the sense of
distributions and satisfies
\[
 \sup_{0\le t\le T}\frac1N
 \int\sum_{i=1}^N(1+|x_i|^2)F^N(t,X^{1:N})\dd X^{1:N}<\infty,
\]
and, for almost every \(t\in[0,T]\),
\begin{align}
 &\int F^N(t)\log\frac{F^N(t)}{G_N}\dd X^{1:N}
 +\frac1\beta\int_0^t\!\int F^N(r)
   \left|\nabla\log\frac{F^N(r)}{G_N}\right|^2
   \dd X^{1:N}\dd r                                                \notag\\
 &\le \int F^N(0)\log\frac{F^N(0)}{G_N}\dd X^{1:N}<\infty.
 \label{eq:particle-free-energy-dissipation}
\end{align}
\end{definition}

For a family of probability densities \((\rho_t)_{0\le t\le T}\) and an
entropy solution \(F^N\), define the modulated free energy by
\[
 \mathcal E_N(t)
 :=\frac1N\Ent(F^N(t)\mid\rho_t^{\otimes N})
 +\frac\beta N\E_{F^N(t)}\cU_{N,\rho_t}(g).
\]
To express \(\mathcal E_N\) as a weighted relative entropy, set
\[
 G_{\rho_t,N}(X^{1:N})
 :=\exp\left[-\beta\sum_{i=1}^N(g*\rho_t)(x_i)
       +\frac{\beta N}{2}\iint g(x-y)\rho_t(x)\rho_t(y)\dd x\dd y
       \right].
\]
Then
\begin{equation}\label{eq:modulated-weight-identity}
 \mathcal E_N(t)
 =\frac1N\int F^N(t)
 \log\left(
 \frac{F^N(t)/G_N}{\rho_t^{\otimes N}/G_{\rho_t,N}}
 \right)\dd X^{1:N}.
\end{equation}

\begin{theorem}
\label{thm:optimal-modulated-free-energy}
Let \((\rho_t)_{0\le t\le T}\) be a classical solution in probability densities
of \eqref{eq:first-order-coulomb-pde}, and set
\(\psi_t:=\log\rho_t+\beta g*\rho_t\).  Suppose that
\[
 \sup_{0\le t\le T}\left(
 \norm{\rho_t}_{L^\infty}
 +\norm{D^2\psi_t}_{L^\infty}
 \right)<\infty.
\]
Let \(F^N\) be an entropy solution in the sense of
Definition~\ref{def:coulomb-entropy-solution}.  There exists \(C_T\ge1\),
independent of \(N\), such that, for every \(0\le t\le T\),
\begin{equation}\label{eq:optimal-modulated-rate}
 \frac1{2N}\Ent(F^N(t)\mid\rho_t^{\otimes N})-\frac{C_T}{N}
 \le\mathcal E_N(t)
 \le e^{C_T t}\mathcal E_N(0)+\frac{C_T}{N}.
\end{equation}
If \(F^N(0)=\rho_0^{\otimes N}\), then
\begin{equation}\label{eq:optimal-product-modulated-rate}
 \sup_{0\le t\le T}\left(
 \abs{\mathcal E_N(t)}
 +\frac1N\Ent(F^N(t)\mid\rho_t^{\otimes N})
 \right)
 \le\frac{C_T}{N}.
\end{equation}
\end{theorem}

We first derive the evolution inequality for \(\mathcal E_N\).

\begin{proposition}
\label{prop:coulomb-modulated-identity}
Suppose that \((\rho_t)_{0\le t\le T}\) and \(F^N\) satisfy the assumptions
of Theorem~\ref{thm:optimal-modulated-free-energy}.  Then
\begin{align}
 \mathcal E_N(t)
 &+\frac1{\beta N}\int_0^t\!\int F^N(r)
 \left|\nabla_{X^{1:N}}\log\left(
 \frac{F^N(r)/G_N}{\rho_r^{\otimes N}/G_{\rho_r,N}}
 \right)\right|^2\dd X^{1:N}\dd r                       \notag\\
 &\le\mathcal E_N(0)
 -\frac1N\int_0^t
  \E_{F^N(r)}\cU_{N,\rho_r}\!\left(
  q_{\nabla\psi_r}\right)\dd r,
 \qquad 0\le t\le T,                                   \label{eq:BJW-inequality}
\end{align}
where
\(q_{\nabla\psi_r}(x,y):=(\nabla\psi_r(x)-\nabla\psi_r(y))
\cdot\nabla g(x-y)\).
\end{proposition}

\begin{proof}
We first assume that \(g\) is smooth and even and that \(F^N\) and \(\rho_t\)
are smooth and positive.  Since \(G_N\) is independent of
time, the \(N\)-particle Fokker--Planck equation
\eqref{eq:first-order-coulomb-fokker-planck} gives
\begin{equation}\label{eq:particle-weighted-equation}
 \partial_t\left(\frac{F^N}{G_N}\right)
 =\frac1{\beta G_N}\nabla_{X^{1:N}}\!\cdot
   \left(G_N\nabla_{X^{1:N}}\frac{F^N}{G_N}\right).
\end{equation}
To compare \(F^N/G_N\) with
\(\rho_t^{\otimes N}/G_{\rho_t,N}\), we substitute
\(\rho_t^{\otimes N}/G_{\rho_t,N}\) into the weighted equation
\eqref{eq:particle-weighted-equation}.  The definition of \(\psi_t\) rewrites
the nonlinear Fokker--Planck equation
\eqref{eq:first-order-coulomb-pde} as
\(\partial_t\rho_t=\beta^{-1}\nabla\!\cdot(\rho_t\nabla\psi_t)\).
The definitions of \(G_N\) and \(G_{\rho_t,N}\) then give
\begin{equation}\label{eq:reference-spatial-identities}
\begin{aligned}
 \nabla_{x_i}\log\left(
 \frac{\rho_t^{\otimes N}}{G_{\rho_t,N}}\right)
 =\nabla\psi_t(x_i),
 \qquad
 \frac1\beta\nabla_{x_i}\log G_N
 &=-\frac1{N-1}\sum_{j\ne i}\nabla g(x_i-x_j),\\
 \frac1\beta\nabla_{x_i}\log G_{\rho_t,N}
 &=-\nabla(g*\rho_t)(x_i).
\end{aligned}
\end{equation}
Differentiating \(G_{\rho_t,N}\) in time and using
\eqref{eq:first-order-coulomb-pde}, we obtain
\begin{equation}\label{eq:reference-time-identity}
\begin{aligned}
 \partial_t\log G_{\rho_t,N}
 &=-\sum_{i=1}^N\int
   \nabla g(x_i-y)\cdot\nabla\psi_t(y)\rho_t(y)\dd y\\
 &\quad-N\int
   \nabla(g*\rho_t)(x)\cdot\nabla\psi_t(x)\rho_t(x)\dd x.
\end{aligned}
\end{equation}
Substituting \eqref{eq:reference-spatial-identities} and
\eqref{eq:reference-time-identity} into
\eqref{eq:particle-weighted-equation}, and using the oddness of \(\nabla g\),
we obtain the following identity:
\begin{align}
 &\frac{G_{\rho_t,N}}{\rho_t^{\otimes N}}
 \left\{
 \partial_t\left(\frac{\rho_t^{\otimes N}}{G_{\rho_t,N}}\right)
 -\frac1{\beta G_N}\nabla_{X^{1:N}}\!\cdot
 \left[G_N\nabla_{X^{1:N}}
 \left(\frac{\rho_t^{\otimes N}}{G_{\rho_t,N}}\right)\right]
 \right\}                                                     \notag\\
 &=\frac1{2(N-1)}\sum_{i\ne j}
   q_{\nabla\psi_t}(x_i,x_j)
   -\sum_{i=1}^N\int q_{\nabla\psi_t}(x_i,y)\rho_t(y)\dd y \notag\\
 &\quad+\frac N2\iint q_{\nabla\psi_t}(x,y)
   \rho_t(x)\rho_t(y)\dd x\dd y
 =\cU_{N,\rho_t}(q_{\nabla\psi_t}).
 \label{eq:reference-weighted-residual}
\end{align}
Differentiating the weighted representation
\eqref{eq:modulated-weight-identity} of the modulated free energy,
substituting the weighted equation
\eqref{eq:particle-weighted-equation} and the residual identity
\eqref{eq:reference-weighted-residual}, and integrating by parts, we obtain
\begin{equation}\label{eq:smooth-modulated-free-energy-identity}
 \frac{\dd}{\dd t}\mathcal E_N(t)
 =-\frac1{\beta N}\int F^N
   \left|\nabla_{X^{1:N}}\log\left(
   \frac{F^N/G_N}{\rho_t^{\otimes N}/G_{\rho_t,N}}
   \right)\right|^2\dd X^{1:N}
  -\frac1N\E_{F^N}\cU_{N,\rho_t}(q_{\nabla\psi_t}).
\end{equation}
Integrating \eqref{eq:smooth-modulated-free-energy-identity} in time, we obtain
\eqref{eq:BJW-inequality} with equality.

For an entropy solution, the modulated free energy inequality
\eqref{eq:BJW-inequality} follows from the weak formulation and the
dissipation estimate \eqref{eq:particle-free-energy-dissipation} in
Definition~\ref{def:coulomb-entropy-solution}; see
\cite[Proposition~2.3]{BreschJabinWang2023} for the analogous argument in the
attractive logarithmic-kernel case.
\end{proof}

\begin{proof}[Proof of Theorem~\ref{thm:optimal-modulated-free-energy}]
Throughout the proof, \(C_T\) denotes a constant independent of \(N\) and
\(t\), whose value may change from line to line.  Using the uniform
\(L^\infty\) bound on \(\rho_t\) and
Theorem~\ref{thm:uniform-partition} at \(\beta\) and \(2\beta\), we obtain
\[
 \sup_{N\ge2}\sup_{0\le t\le T}
 \bigl(\log\cZ_{N,\rho_t}(\beta;g)
       +\log\cZ_{N,\rho_t}(2\beta;g)+1\bigr)\le C_T.
\]

We first compare the modulated free energy with the relative entropy.
Applying the Donsker--Varadhan inequality to
\(-2\beta\cU_{N,\rho_t}(g)\), we obtain
\[
 -2\beta\E_{F^N(t)}\cU_{N,\rho_t}(g)
 \le\Ent(F^N(t)\mid\rho_t^{\otimes N})
   +\log\cZ_{N,\rho_t}(2\beta;g).
\]
Consequently,
\begin{equation}\label{eq:dynamic-entropy-coercivity}
 \mathcal E_N(t)
 \ge\frac1{2N}\Ent(F^N(t)\mid\rho_t^{\otimes N})
 -\frac{\log\cZ_{N,\rho_t}(2\beta;g)}{2N}
 \ge\frac1{2N}\Ent(F^N(t)\mid\rho_t^{\otimes N})-\frac{C_T}{N}.
\end{equation}
In particular, \(\mathcal E_N(t)+C_T/N\ge0\).

We next estimate the error term in the evolution inequality
\eqref{eq:BJW-inequality}.  By the definition of the centered Gibbs measure
\(\cP_{N,\rho_t}^{\beta}\) in \eqref{eq:centered-gibbs-measure},
\begin{equation}\label{eq:dynamic-entropy-gibbs-identity}
 \frac1N\Ent(F^N(t)\mid\cP_{N,\rho_t}^{\beta})
 =\mathcal E_N(t)
 +\frac{\log\cZ_{N,\rho_t}(\beta;g)}N.
\end{equation}
The bound on \(D^2\psi_t\) allows us to apply
Proposition~\ref{prop:gibbs-commutator-exponential}, specifically
\eqref{eq:average-coulomb-commutator}, with \(u=\nabla\psi_t\).  Using the
entropy--Gibbs identity \eqref{eq:dynamic-entropy-gibbs-identity}, we obtain
\begin{equation}\label{eq:dynamic-commutator-control}
 \frac1N\left|\E_{F^N(t)}\cU_{N,\rho_t}\!\left(
 q_{\nabla\psi_t}\right)\right|
 \le\frac1N\E_{F^N(t)}\left|\cU_{N,\rho_t}\!\left(
 q_{\nabla\psi_t}\right)\right|
 \le C_T\left(\mathcal E_N(t)+\frac{C_T}{N}\right).
\end{equation}

Substituting the commutator estimate
\eqref{eq:dynamic-commutator-control} into the modulated free energy
inequality \eqref{eq:BJW-inequality} and dropping the nonnegative
dissipation term, we find
\[
 \mathcal E_N(t)
 \le\mathcal E_N(0)
 +C_T\int_0^t\left(\mathcal E_N(r)+\frac{C_T}{N}\right)\dd r.
\]
Applying Gronwall's inequality, we obtain
\begin{equation}\label{eq:optimal-M-rate}
 \mathcal E_N(t)
 \le e^{C_T t}\mathcal E_N(0)+\frac{C_T}{N}.
\end{equation}
Combining \eqref{eq:optimal-M-rate} with the coercivity bound
\eqref{eq:dynamic-entropy-coercivity} proves
\eqref{eq:optimal-modulated-rate}.

If \(F^N(0)=\rho_0^{\otimes N}\), then
\(\mathcal E_N(0)=0\) because
\(\E_{\rho_0^{\otimes N}}\cU_{N,\rho_0}(g)=0\).  Hence
\eqref{eq:optimal-product-modulated-rate} follows from
\eqref{eq:optimal-modulated-rate}.
\end{proof}

\section{Application III: Quantitative Gaussian fluctuations}
\label{sec:gaussian-fluctuations}

In this section, we use a uniform partition-function estimate to control
the nonlinear term in the fluctuation equation and prove a quantitative
Gaussian fluctuation result.

We consider the first-order mean-field diffusion for any kernel
\(g=g_{d,s}^{(m)}\) covered by Theorem~\ref{thm:uniform-partition}.  The
three-dimensional Coulomb model treated in
Section~\ref{sec:optimal-modulated-free-energy} is recovered by taking
\(g=g_{3,1}\).  Fix \(\beta>0\).  For \(N\ge2\), let
\(X^{N,1},\ldots,X^{N,N}\) solve
\begin{equation}\label{eq:fluctuation-particle-system}
 \dd X_t^{N,i}
 =-\frac1{N-1}\sum_{j\ne i}\nabla g(X_t^{N,i}-X_t^{N,j})\dd t
   +\sqrt{\frac2\beta}\dd B_t^i,
 \qquad 1\le i\le N,
\end{equation}
where \(B^1,\ldots,B^N\) are independent standard Brownian motions in
\(\mathbb R^d\).  Fix \(T>0\), and let \((\rho_t)_{0\le t\le T}\) be a probability
density solution of the corresponding mean-field equation
\begin{equation}\label{eq:fluctuation-limit-equation}
 \partial_t\rho_t
 -\nabla\!\cdot\bigl(\rho_t\nabla(g*\rho_t)\bigr)
 =\frac1\beta\Delta\rho_t.
\end{equation}
We assume that \(X_0^{N,1},\ldots,X_0^{N,N}\) are independent with common
density \(\rho_0\) and are independent of the Brownian motions.

Define the empirical measure and the fluctuation measure by
\[
 \mu_t^N:=\frac1N\sum_{i=1}^N\delta_{X_t^{N,i}},
 \qquad
 \eta_t^N:=\sqrt N\,(\mu_t^N-\rho_t).
\]
As in Section~\ref{sec:optimal-modulated-free-energy}, write
\[
 q_u(x,y)=(u(x)-u(y))\cdot\nabla g(x-y).
\]
For \(\varphi\in C_c^\infty(\mathbb R^d)\), applying It\^o's formula to
\eqref{eq:fluctuation-particle-system} and then symmetrizing, we obtain
\begin{align}
 \dd\langle\eta_t^N,\varphi\rangle
 &=\frac1\beta\langle\eta_t^N,\Delta\varphi\rangle\dd t
 -\langle\eta_t^N,
   \nabla\varphi\cdot\nabla(g*\rho_t)\rangle\dd t
 -\langle\rho_t,
   \nabla\varphi\cdot(\nabla g*\eta_t^N)\rangle\dd t
 \notag\\
 &\quad-\frac1{2\sqrt N}\iint_{x\ne y}q_{\nabla\varphi}(x,y)
   \eta_t^N(\dd x)\eta_t^N(\dd y)\dd t
 \notag\\
 &\quad-\frac1{\sqrt N\,N(N-1)}\sum_{i<j}
   q_{\nabla\varphi}(X_t^{N,i},X_t^{N,j})\dd t
 +\dd M_{N,t}^{\varphi},                              \label{eq:finite-fluctuation-equation}
\end{align}
The correction term in the last line of
\eqref{eq:finite-fluctuation-equation} comes from the normalization
\(1/(N-1)\) in \eqref{eq:fluctuation-particle-system}, rather than
\(1/N\).  We denote the martingale term by
\[
 M_{N,t}^{\varphi}
 :=\sqrt{\frac2{\beta N}}\sum_{i=1}^N
   \int_0^t\nabla\varphi(X_r^{N,i})\cdot\dd B_r^i.
\]

These two drift terms form the nonlinear remainder, which can be written in
terms of the centered commutator associated with \(\nabla\varphi\) as
\begin{align*}
 &-\frac1{2\sqrt N}\iint_{x\ne y}q_{\nabla\varphi}(x,y)
   \eta_t^N(\dd x)\eta_t^N(\dd y)
   \dd t
   -\frac1{\sqrt N\,N(N-1)}\sum_{i<j}
   q_{\nabla\varphi}(X_t^{N,i},X_t^{N,j})\dd t\\
 &\qquad=-\frac1{\sqrt N\,(N-1)}\sum_{i<j}
   (q_{\nabla\varphi})_{\rho_t}^{\circ}
   (X_t^{N,i},X_t^{N,j})\dd t\\
 &\qquad=-\frac1{\sqrt N}
   \cU_{N,\rho_t}(q_{\nabla\varphi})\dd t.
\end{align*}
To estimate this commutator, let \(F^N(t)\) be the density of the joint law of
\(X_t^{N,1},\ldots,X_t^{N,N}\), and set
\[
 \psi_t:=\log\rho_t+\beta g*\rho_t,
 \qquad
 \mathcal E_N(t)
 :=\frac1N\Ent(F^N(t)\mid\rho_t^{\otimes N})
   +\frac\beta N\E_{F^N(t)}\cU_{N,\rho_t}(g).
\]

\begin{lemma}
\label{lem:dynamic-fluctuation-estimates}
Suppose that \(\rho_t\) is a positive classical solution of
\eqref{eq:fluctuation-limit-equation} and that
\[
 \sup_{0\le t\le T}\left(
 \norm{\rho_t}_{L^\infty}+\norm{D^2\psi_t}_{L^\infty}
 \right)<\infty.
\]
When \(s=0\), suppose in addition that, for some \(a,q>0\),
\[
 \sup_{0\le t\le T}\int_{\mathbb R^d}e^{a|x|^q}\rho_t(x)\dd x<\infty.
\]
Suppose that \(F^N\) is an entropy solution in the sense of
Definition~\ref{def:coulomb-entropy-solution}, with \(\mathbb R^3\) and
\(g_{3,1}\) replaced by \(\mathbb R^d\) and \(g\), respectively.  Then the following
estimates hold.
\begin{enumerate}
\item[(i)] For every family of globally Lipschitz vector fields
\((u_t)_{0\le t\le T}\) with
\(\sup_t\norm{Du_t}_\infty<\infty\),
\begin{equation}\label{eq:uniform-nonlinear-fluctuation-bound}
 \sup_{N\ge2}\sup_{0\le t\le T}
 \E_{F^N(t)}\abs{\cU_{N,\rho_t}(q_{u_t})}<\infty.
\end{equation}
\item[(ii)] The modulated free energy satisfies
\begin{equation}\label{eq:general-modulated-free-energy-inequality}
 \mathcal E_N(t)
 \le \mathcal E_N(0)
 -\frac1N\int_0^t
   \E_{F^N(r)}\cU_{N,\rho_r}(q_{\nabla\psi_r})\dd r,
 \qquad 0\le t\le T.
\end{equation}
\end{enumerate}
\end{lemma}

\begin{proof}
The proof is the same as in Section~\ref{sec:optimal-modulated-free-energy},
using Propositions~\ref{prop:gibbs-commutator-exponential}
and~\ref{prop:coulomb-modulated-identity}, and is omitted.
\end{proof}

The remaining estimate concerns the linear statistics of the fluctuation
measure and will be used to control the martingale covariation.

\begin{lemma}
\label{lem:linear-fluctuation-moment}
Under the assumptions of Lemma~\ref{lem:dynamic-fluctuation-estimates}, for
every bounded measurable function \(\varphi\),
\begin{equation}\label{eq:linear-fluctuation-moment}
 \sup_{N\ge2}\sup_{0\le t\le T}
 \E\abs{\langle\eta_t^N,\varphi\rangle}
 \le C_T\norm{\varphi}_\infty.
\end{equation}
\end{lemma}

\begin{proof}
Since \(F^N(0)=\rho_0^{\otimes N}\) and
\(\E_{\rho_0^{\otimes N}}\cU_{N,\rho_0}(g)=0\), we have
\(\mathcal E_N(0)=0\).  Using
\eqref{eq:uniform-nonlinear-fluctuation-bound} with
\(u_t=\nabla\psi_t\) and
\eqref{eq:general-modulated-free-energy-inequality}, we obtain
\[
 \sup_{N\ge2}\sup_{0\le t\le T}N\mathcal E_N(t)\le C_T.
\]
The entropy--Gibbs identity
\eqref{eq:dynamic-entropy-gibbs-identity} is unchanged when \(\mathbb R^3\) and
\(g_{3,1}\) are replaced by \(\mathbb R^d\) and \(g\), respectively.  Thus
\[
 \Ent(F^N(t)\mid\cP_{N,\rho_t}^{\beta})
 =N\mathcal E_N(t)+\log\cZ_{N,\rho_t}(\beta;g).
\]
Using Theorem~\ref{thm:uniform-partition}, we obtain
\begin{equation}\label{eq:uniform-gibbs-entropy}
 \sup_{N\ge2}\sup_{0\le t\le T}
 \Ent(F^N(t)\mid\cP_{N,\rho_t}^{\beta})\le C_T.
\end{equation}
If \(\norm{\varphi}_\infty=0\), there is nothing to prove.  Otherwise, set
\(\bar\varphi=\varphi/\norm{\varphi}_\infty\) and
\[
 \langle\eta_t^N,\bar\varphi\rangle
 =\frac1{\sqrt N}\sum_{i=1}^N
 \bigl(\bar\varphi(X_t^{N,i})-\langle\rho_t,\bar\varphi\rangle\bigr).
\]
By the definition of the centered Gibbs measure,
\[
 \E_{\cP_{N,\rho_t}^{\beta}}
 e^{|\langle\eta_t^N,\bar\varphi\rangle|}
 =\frac1{\cZ_{N,\rho_t}(\beta;g)}
  \E_{\rho_t^{\otimes N}}
  e^{-\beta\cU_{N,\rho_t}(g)
     +|\langle\eta_t^N,\bar\varphi\rangle|}.
\]
Applying the Cauchy--Schwarz inequality and using
Theorem~\ref{thm:uniform-partition}, we obtain
\[
 \E_{\cP_{N,\rho_t}^{\beta}}
 e^{|\langle\eta_t^N,\bar\varphi\rangle|}
 \le C_T\left(\E_{\rho_t^{\otimes N}}
 e^{2|\langle\eta_t^N,\bar\varphi\rangle|}\right)^{1/2}.
\]
Under \(\rho_t^{\otimes N}\),
\(\langle\eta_t^N,\bar\varphi\rangle\) is the normalized sum of centered,
bounded i.i.d.\ random variables.  Applying Hoeffding's inequality, we obtain
\[
 \E_{\rho_t^{\otimes N}}
 e^{\pm2\langle\eta_t^N,\bar\varphi\rangle}\le e^2.
\]
Using \(e^{2|z|}\le e^{2z}+e^{-2z}\), we conclude that
\[
 \sup_{N\ge2}\sup_{0\le t\le T}
 \E_{\cP_{N,\rho_t}^{\beta}}
 e^{|\langle\eta_t^N,\bar\varphi\rangle|}\le C_T.
\]
The entropy variational inequality, together with this bound and
\eqref{eq:uniform-gibbs-entropy}, gives
\[
 \E_{F^N(t)}|\langle\eta_t^N,\bar\varphi\rangle|
 \le \Ent(F^N(t)\mid\cP_{N,\rho_t}^{\beta})
  +\log\E_{\cP_{N,\rho_t}^{\beta}}
   e^{|\langle\eta_t^N,\bar\varphi\rangle|}
 \le C_T.
\]
Multiplying by \(\norm{\varphi}_\infty\), we arrive at
\eqref{eq:linear-fluctuation-moment}.
\end{proof}

We now establish quantitative Gaussian fluctuations for fixed-time
finite-dimensional projections of \(\eta_t^N\).  For a smooth function
\(\varphi\), set
\[
 (\mathcal L_t\varphi)(x)
 :=\frac1\beta\Delta\varphi(x)
   -\nabla(g*\rho_t)(x)\cdot\nabla\varphi(x)
   +\int\nabla g(x-y)\cdot\nabla\varphi(y)\rho_t(y)\dd y.
\]
With this notation, \eqref{eq:finite-fluctuation-equation} takes the concise
form
\begin{equation}\label{eq:concise-fluctuation-equation}
 \dd\langle\eta_t^N,\varphi\rangle
 =\langle\eta_t^N,\mathcal L_t\varphi\rangle\dd t
  -\frac1{\sqrt N}\cU_{N,\rho_t}(q_{\nabla\varphi})\dd t
  +\dd M_{N,t}^{\varphi}.
\end{equation}
Fix \(t\in[0,T]\), \(J\ge1\), and
\(\boldsymbol\varphi=(\varphi_1,\ldots,\varphi_J)
\in(C_c^\infty(\mathbb R^d))^J\).  Consider the vector-valued backward equation
\begin{equation}\label{eq:backward-fluctuation-equation}
 \partial_r\boldsymbol\varphi_r
 +\mathcal L_r\boldsymbol\varphi_r=0,
 \qquad \boldsymbol\varphi_t=\boldsymbol\varphi.
\end{equation}
We denote the finite-dimensional fluctuation vector by
\[
 Y_{N,t}:=\langle\eta_t^N,\boldsymbol\varphi\rangle
 =\bigl(\langle\eta_t^N,\varphi_1\rangle,\ldots,
          \langle\eta_t^N,\varphi_J\rangle\bigr).
\]

\begin{theorem}
\label{thm:quantitative-gaussian-fluctuations}
Let \(\rho_t\) and \(F^N(t)\) satisfy the assumptions of
Lemma~\ref{lem:dynamic-fluctuation-estimates}.  Suppose that
the linear backward equation \eqref{eq:backward-fluctuation-equation} admits a
unique classical solution
\((\boldsymbol\varphi_r)_{0\le r\le t}\) satisfying
\begin{equation}
 \sup_{0\le r\le t}
 \left(\norm{\boldsymbol\varphi_r}_{C_b^2}
       +\norm{\partial_r\boldsymbol\varphi_r}_\infty\right)<\infty.
 \label{eq:backward-fluctuation-regularity}
\end{equation}
Let
\[
 Y_0\sim\mathcal N\!\left(
 0,
 \left\langle\rho_0,
 \boldsymbol\varphi_0\boldsymbol\varphi_0^{\mathsf T}\right\rangle
 -\left\langle\rho_0,\boldsymbol\varphi_0\right\rangle
  \left\langle\rho_0,\boldsymbol\varphi_0\right\rangle^{\mathsf T}
 \right),
\]
and let \(W\) be a standard \(J\)-dimensional Brownian motion independent of
\(Y_0\).  Define the centered Gaussian vector
\begin{equation}\label{eq:dynamic-gaussian-representation}
 Y_t:=Y_0+\int_0^t
 \left(\frac2\beta
 \left\langle\rho_r,
 D\boldsymbol\varphi_r(D\boldsymbol\varphi_r)^{\mathsf T}
 \right\rangle\right)^{1/2}\dd W_r.
\end{equation}
There exists a constant
\(C_{t,\boldsymbol\varphi}<\infty\), independent of \(N\), such that
\begin{equation}\label{eq:gaussian-fluctuation-estimate}
 \left|\E\Phi\left(\langle\eta_t^N,\boldsymbol\varphi\rangle\right)
       -\E\Phi(Y_t)\right|
 \le\frac{C_{t,\boldsymbol\varphi}}{\sqrt N}
 \left(\norm{\nabla\Phi}_\infty
       +\norm{\nabla^2\Phi}_\infty\right),
 \qquad \Phi\in C_b^2(\mathbb R^J).
\end{equation}
Consequently, \(Y_{N,t}\) converges in law to \(Y_t\) as \(N\to\infty\).
\end{theorem}

\begin{proof}
Using \eqref{eq:concise-fluctuation-equation} and
\eqref{eq:backward-fluctuation-equation}, we obtain
\begin{equation}\label{eq:vector-fluctuation-dynamics}
\begin{aligned}
 \dd\langle\eta_r^N,\boldsymbol\varphi_r\rangle
 ={}&-\frac1{\sqrt N}
 \cU_{N,\rho_r}(q_{\nabla\boldsymbol\varphi_r})\dd r\\
 &+\sqrt{\frac2{\beta N}}\sum_{i=1}^N
 D\boldsymbol\varphi_r(X_r^{N,i})\dd B_r^i.
\end{aligned}
\end{equation}
Write
\[
 \sigma_r:=\left(\frac2\beta
 \left\langle\rho_r,
 D\boldsymbol\varphi_r(D\boldsymbol\varphi_r)^{\mathsf T}
 \right\rangle\right)^{1/2}.
\]
Then the quadratic covariation density of the martingale term in
\eqref{eq:vector-fluctuation-dynamics} is
\begin{equation}\label{eq:vector-martingale-covariation}
 \frac2\beta
 \left\langle\mu_r^N,
 D\boldsymbol\varphi_r(D\boldsymbol\varphi_r)^{\mathsf T}\right\rangle
 =\sigma_r\sigma_r^{\mathsf T}+\frac2{\beta\sqrt N}
 \left\langle\eta_r^N,
 D\boldsymbol\varphi_r(D\boldsymbol\varphi_r)^{\mathsf T}\right\rangle.
\end{equation}
For \(\Phi\in C_b^2(\mathbb R^J)\), define
\[
 u_r(x):=\E\left[
 \Phi\left(x+\int_r^t\sigma_s\dd W_s\right)\right].
\]
Then \(u\) solves
\begin{equation}\label{eq:finite-dimensional-gaussian-backward-equation}
 \partial_ru_r+\frac12\operatorname{Tr}
 \left(\sigma_r\sigma_r^{\mathsf T}\nabla^2u_r\right)=0,
 \qquad u_t=\Phi,
\end{equation}
and differentiation under the expectation gives
\begin{equation}\label{eq:finite-dimensional-gaussian-derivative-bounds}
 \sup_{0\le r\le t}\norm{\nabla u_r}_\infty
 \le\norm{\nabla\Phi}_\infty,
 \qquad
 \sup_{0\le r\le t}\norm{\nabla^2u_r}_\infty
 \le\norm{\nabla^2\Phi}_\infty.
\end{equation}
By \eqref{eq:dynamic-gaussian-representation} and the independence of
\(Y_0\) and \(W\), we find that
\[
 \E\Phi(Y_t)=\E u_0(Y_0).
\]

Applying It\^o's formula to
\(u_r(\langle\eta_r^N,\boldsymbol\varphi_r\rangle)\) and using
\eqref{eq:vector-fluctuation-dynamics},
\eqref{eq:vector-martingale-covariation}, and
\eqref{eq:finite-dimensional-gaussian-backward-equation}, we obtain
\begin{align*}
 &\dd u_r\left(
 \langle\eta_r^N,\boldsymbol\varphi_r\rangle\right)\\
 ={}&-\frac1{\sqrt N}
 \nabla u_r\left(
 \langle\eta_r^N,\boldsymbol\varphi_r\rangle\right)
 \cdot\cU_{N,\rho_r}
 (q_{\nabla\boldsymbol\varphi_r})\dd r\\
 &+\frac1{\beta\sqrt N}
 \operatorname{Tr}\left[
 \nabla^2u_r\left(
 \langle\eta_r^N,\boldsymbol\varphi_r\rangle\right)
 \left\langle\eta_r^N,
 D\boldsymbol\varphi_r(D\boldsymbol\varphi_r)^{\mathsf T}\right\rangle
 \right]\dd r\\
 &+\sqrt{\frac2{\beta N}}\sum_{i=1}^N
 \nabla u_r\left(
 \langle\eta_r^N,\boldsymbol\varphi_r\rangle\right)^{\mathsf T}
 D\boldsymbol\varphi_r(X_r^{N,i})\dd B_r^i.
\end{align*}
The equation \eqref{eq:finite-dimensional-gaussian-backward-equation} for
\(u\) cancels the term containing
\(\sigma_r\sigma_r^{\mathsf T}\).  The two remaining finite-variation terms
are the nonlinear drift in \eqref{eq:vector-fluctuation-dynamics} and the
fluctuation of the quadratic covariation in
\eqref{eq:vector-martingale-covariation}, respectively.  By
\eqref{eq:finite-dimensional-gaussian-derivative-bounds} and
\eqref{eq:backward-fluctuation-regularity}, the stochastic integral is a
martingale with zero expectation.  Since \(u_t=\Phi\) and
\(\boldsymbol\varphi_t=\boldsymbol\varphi\), integrating over \([0,t]\),
taking expectations, and using \eqref{eq:dynamic-gaussian-representation},
we obtain
\begin{align}
 &\abs{\E\Phi(Y_{N,t})-\E\Phi(Y_t)}
 \notag\\
 \le{}&
 \abs{\E u_0\left(
 \langle\eta_0^N,\boldsymbol\varphi_0\rangle\right)
 -\E u_0(Y_0)}
 \notag\\
 &+\frac1{\sqrt N}\E\int_0^t
 \abs{\nabla u_r\left(
 \langle\eta_r^N,\boldsymbol\varphi_r\rangle\right)
 \cdot\cU_{N,\rho_r}(q_{\nabla\boldsymbol\varphi_r})}\dd r
 \notag\\
 &+\frac1{\beta\sqrt N}\E\int_0^t
 \abs{\operatorname{Tr}\left[
 \nabla^2u_r\left(
 \langle\eta_r^N,\boldsymbol\varphi_r\rangle\right)
 \left\langle\eta_r^N,
 D\boldsymbol\varphi_r(D\boldsymbol\varphi_r)^{\mathsf T}\right\rangle
 \right]}\dd r.
 \label{eq:finite-dimensional-three-errors}
\end{align}

For the initial term,
\[
 \langle\eta_0^N,\boldsymbol\varphi_0\rangle
 =\frac1{\sqrt N}\sum_{i=1}^N
 \left(\boldsymbol\varphi_0(X_0^{N,i})
       -\langle\rho_0,\boldsymbol\varphi_0\rangle\right)
\]
is a normalized sum of bounded, centered i.i.d.\ vectors with the same
covariance as \(Y_0\).  Applying the multivariate \(W_1\) Berry--Esseen
estimate \cite[Theorem~2.15 and Equation~(3.5)]{Raic2019} on the range of
this covariance and using the first estimate in
\eqref{eq:finite-dimensional-gaussian-derivative-bounds}, we obtain
\begin{equation}\label{eq:initial-gaussian-error}
\begin{aligned}
 &\left|\E u_0\left(
 \langle\eta_0^N,\boldsymbol\varphi_0\rangle\right)
 -\E u_0(Y_0)\right|\\
 &\qquad\le\frac{C_{t,\boldsymbol\varphi}}{\sqrt N}
 \norm{\nabla u_0}_\infty
 \le\frac{C_{t,\boldsymbol\varphi}}{\sqrt N}
 \norm{\nabla\Phi}_\infty.
\end{aligned}
\end{equation}

Next, for the nonlinear drift term, applying
Lemma~\ref{lem:dynamic-fluctuation-estimates} and using the first estimate in
\eqref{eq:finite-dimensional-gaussian-derivative-bounds}, we obtain
\begin{equation}\label{eq:nonlinear-drift-error}
 \frac1{\sqrt N}\E\int_0^t
 \abs{\nabla u_r\left(
 \langle\eta_r^N,\boldsymbol\varphi_r\rangle\right)
 \cdot\cU_{N,\rho_r}(q_{\nabla\boldsymbol\varphi_r})}\dd r
 \le\frac{C_{t,\boldsymbol\varphi}}{\sqrt N}
 \norm{\nabla\Phi}_\infty.
\end{equation}
Finally, for the quadratic covariation term, using
Lemma~\ref{lem:linear-fluctuation-moment}, the second estimate in
\eqref{eq:finite-dimensional-gaussian-derivative-bounds}, and the trace
inequality, we obtain
\begin{equation}\label{eq:covariation-error}
\begin{aligned}
 &\frac1{\beta\sqrt N}\E\int_0^t
 \abs{\operatorname{Tr}\left[
 \nabla^2u_r\left(
 \langle\eta_r^N,\boldsymbol\varphi_r\rangle\right)
 \left\langle\eta_r^N,
 D\boldsymbol\varphi_r(D\boldsymbol\varphi_r)^{\mathsf T}\right\rangle
 \right]}\dd r\\
 &\qquad\le\frac{C_{t,\boldsymbol\varphi}}{\sqrt N}
 \norm{\nabla^2\Phi}_\infty.
\end{aligned}
\end{equation}
Substituting \eqref{eq:initial-gaussian-error},
\eqref{eq:nonlinear-drift-error}, and \eqref{eq:covariation-error} into
\eqref{eq:finite-dimensional-three-errors} proves
\eqref{eq:gaussian-fluctuation-estimate}.  For every \(\xi\in\mathbb R^J\), applying
this estimate to \(x\mapsto\cos(\xi\cdot x)\) and
\(x\mapsto\sin(\xi\cdot x)\), we obtain convergence of the characteristic
functions.  The convergence in law of \(Y_{N,t}\) to \(Y_t\) then follows by
applying L\'evy's continuity theorem.
\end{proof}

\begin{remark}
Standard estimates for linear backward parabolic equations provide the
well-posedness of \eqref{eq:backward-fluctuation-equation} and the regularity
in \eqref{eq:backward-fluctuation-regularity}; see
\cite[Lemma~3.21]{WangZhaoZhu2023} and
\cite[Lemma~4.9]{HaoZhangZhao2026}.
\end{remark}

\section{The determinant limit and sharp dependence on
\texorpdfstring{\(\beta\)}{beta}}
\label{sec:second-order-limits}

We finally return to the centered partition function.  Combining the second
Wiener chaos limit with the uniform bound, we prove convergence of the Laplace
transforms and identify the limit as a Carleman--Fredholm determinant.  We then
use the determinant formula to establish the sharp dependence on \(\beta\).

\subsection{Second Wiener chaos and the determinant limit}

We first establish the properties of \(A_{g,\rho}\) needed for its spectral
representation.

\begin{lemma}\label{lem:kernel-operator-properties}
Under Assumption~\ref{ass:kernel-density}, let \(g=g_{d,s}^{(m)}\).  Then
\(g_\rho^\circ\in L^2(\rho^{\otimes2})\), and \(A_{g,\rho}\) is a
nonnegative self-adjoint Hilbert--Schmidt operator satisfying
\(A_{g,\rho}\mathbf1=0\).
\end{lemma}

\begin{proof}
Recall that \(A_{g,\rho}\) is the integral operator with kernel
\(g_\rho^\circ\):
\[
 (A_{g,\rho}f)(x)
 :=\int_{\mathbb R^d}g_\rho^\circ(x,y)f(y)\rho(y)\dd y.
\]
The symmetry of \(g_\rho^\circ\) and the cancellation
\eqref{eq:canonical-centering} imply that \(A_{g,\rho}\) is self-adjoint and
\(A_{g,\rho}\mathbf1=0\).  It remains to prove that it is Hilbert--Schmidt
and nonnegative.

We first prove the square integrability of its kernel.  If \(0<s<d/2\), using
the heat representation
\eqref{eq:riesz-split} and \(0\le e^{-m^2t}\le1\), we obtain
\(0\le g=g_{d,s}^{(m)}\le g_{d,s}=c_{d,s}|\cdot|^{-s}\).  Therefore,
\[
 \iint g(x-y)^2\rho(x)\rho(y)\dd x\dd y
 \le C_{d,s}\left(
 \norm{\rho}_\infty\int_{|z|\le1}|z|^{-2s}\dd z+1\right)<\infty.
\]
If \(s=0\), using the local square integrability of \(\log|\cdot|\) and
\eqref{eq:log-tail-assumption}, we obtain the same bound.  The definition of
the centering and Jensen's inequality imply
\[
 \norm{g_\rho^\circ}_{L^2(\rho^{\otimes2})}
 \le4\norm{g(x-y)}_{L^2(\rho^{\otimes2})}<\infty.
\]
Using the Hilbert--Schmidt criterion for integral operators, we obtain
\[
 \norm{A_{g,\rho}}_{\mathrm{HS}}^2
 =\iint |g_\rho^\circ(x,y)|^2\rho(x)\rho(y)\dd x\dd y<\infty.
\]

It remains to prove nonnegativity.  Let \(g=g_\tau+r_\tau\) be the heat
decomposition \eqref{eq:heat-decomposition}, and let \(A_\tau\) be the
integral operator with kernel \((g_\tau)_\rho^\circ\).  When \(s=0\), the
signed measure
\((f-\langle\rho,f\rangle)\rho\dd x\) has zero mass and satisfies the
moment condition in Lemma~\ref{lem:log-coarse-positive} by
the Cauchy--Schwarz inequality and \eqref{eq:log-tail-assumption}.  Using that lemma when
\(s=0\) and \eqref{eq:riesz-coarse-positive} when \(s>0\), for every
\(f\in L^2(\rho)\) we obtain
\[
 \begin{aligned}
 \langle f,A_\tau f\rangle_{L^2(\rho)}
 &=\iint g_\tau(x-y)
 \bigl(f(x)-\langle\rho,f\rangle\bigr)
 \bigl(f(y)-\langle\rho,f\rangle\bigr)
 \rho(x)\rho(y)\dd x\dd y\ge0.
 \end{aligned}
\]
Thus \(A_\tau\) is nonnegative.  Since
\(A_{g,\rho}-A_\tau\) has kernel \((r_\tau)_\rho^\circ\), the previous
centering estimate yields
\[
 \begin{aligned}
 \norm{A_{g,\rho}-A_\tau}_{\mathrm{op}}
 &\le \norm{A_{g,\rho}-A_\tau}_{\mathrm{HS}}
 =\norm{(r_\tau)_\rho^\circ}_{L^2(\rho^{\otimes2})}\\
 &\le4\left(\iint r_\tau(x-y)^2\rho(x)\rho(y)\dd x\dd y\right)^{1/2}
 \le4\norm{\rho}_\infty^{1/2}\norm{r_\tau}_2.
 \end{aligned}
\]
The right-hand side tends to zero as \(\tau\downarrow0\) by
\eqref{eq:log-remainder-scales} when \(s=0\), and by
\eqref{eq:riesz-remainder-l2} when \(s>0\).  Hence
\(A_\tau\to A_{g,\rho}\) in operator norm, and for every \(f\in L^2(\rho)\),
\[
 \langle f,A_{g,\rho}f\rangle_{L^2(\rho)}
 =\lim_{\tau\downarrow0}\langle f,A_\tau f\rangle_{L^2(\rho)}\ge0.
\]
This proves the nonnegativity of \(A_{g,\rho}\).
\end{proof}

\begin{proof}[Proof of Theorem~\ref{thm:gaussian-determinant}]
Set \(K=g_\rho^\circ\).  With this notation,
\eqref{eq:modulated-energy} becomes
\[
 \cU_{N,\rho}(g;X^{1:N})
 =\frac1{N-1}\sum_{1\le i<j\le N}K(X_i,X_j)
 =\cU_{N,\rho}(K;X^{1:N}).
\]
Lemma~\ref{lem:kernel-operator-properties} shows that \(A_{g,\rho}\) is a
nonnegative self-adjoint Hilbert--Schmidt operator and is therefore compact.
Hence all its eigenvalues are nonnegative.  Applying
\cite[Theorems~VI.16 and~VI.23]{ReedSimon1980}, we list them in nonincreasing
order, repeated according to multiplicity, as
\(\kappa_1\ge\kappa_2\ge\cdots\ge0\), and
choose corresponding orthonormal eigenfunctions \((\phi_j)_{j\ge1}\) such
that
\begin{equation}\label{eq:kernel-spectral-decomposition}
 K=\sum_{j\ge1}\kappa_j\phi_j\otimes\phi_j
 \quad\text{in }L^2(\rho^{\otimes2}),
 \qquad
 \sum_{j\ge1}\kappa_j^2=\norm{K}_{L^2(\rho^{\otimes2})}^2.
\end{equation}
For every \(j\) with \(\kappa_j>0\), self-adjointness and
\(A_{g,\rho}\mathbf1=0\) imply
\[
 \kappa_j\int_{\mathbb R^d}\phi_j\rho
 =\langle A_{g,\rho}\phi_j,\mathbf1\rangle_{L^2(\rho)}
 =\langle\phi_j,A_{g,\rho}\mathbf1\rangle_{L^2(\rho)}=0.
\]
Thus every eigenfunction appearing with a nonzero coefficient has zero
\(\rho\)-mean.

Let \((Z_j)_{j\ge1}\) be independent standard Gaussian variables.  For
\(M\ge1\), set
\[
 K_M:=\sum_{j=1}^M\kappa_j\phi_j\otimes\phi_j,
 \qquad
 \cQ_M:=\frac12\sum_{j=1}^M\kappa_j(Z_j^2-1).
\]
The variables \((Z_j^2-1)_{j\ge1}\) are independent centered chi-square
variables with one degree of freedom, and
\(\E(Z_j^2-1)^2=2\).  Hence \eqref{eq:kernel-spectral-decomposition} implies
that the series
\[
 \cQ_{g,\rho}:=\frac12\sum_{j\ge1}\kappa_j(Z_j^2-1)
\]
converges in \(L^2\), and
\begin{equation}\label{eq:gaussian-chaos-tail}
 \E\abs{\cQ_{g,\rho}-\cQ_M}^2
 =\frac12\sum_{j>M}\kappa_j^2\xrightarrow[M\to\infty]{}0.
\end{equation}
Each \(\cQ_M\) is a finite centered quadratic polynomial in the Gaussian
variables.  Since the second Wiener chaos is the \(L^2\)-closure of such
polynomials, \(\cQ_{g,\rho}\) belongs to the second Wiener chaos.

A direct calculation gives
\begin{equation}\label{eq:finite-spectral-statistic}
 \cU_{N,\rho}(K_M;X^{1:N})
 =\frac{N}{2(N-1)}\sum_{j=1}^M\kappa_j\left[
 \left(\frac1{\sqrt N}\sum_{i=1}^N\phi_j(X_i)\right)^2
 -\frac1N\sum_{i=1}^N\phi_j(X_i)^2\right].
\end{equation}
The eigenfunctions are centered and orthonormal in \(L^2(\rho)\).  Therefore,
for every fixed \(M\), the multivariate central limit theorem and the law of
large numbers imply
\[
 \left(\frac1{\sqrt N}\sum_{i=1}^N\phi_j(X_i)\right)_{1\le j\le M}
 \xrightarrow[N\to\infty]{\mathrm{law}}(Z_j)_{1\le j\le M},
 \qquad
 \frac1N\sum_{i=1}^N\phi_j(X_i)^2
 \xrightarrow[N\to\infty]{\mathbb P}1
\]
for every \(1\le j\le M\).  Substituting these limits into
\eqref{eq:finite-spectral-statistic}, we obtain
\begin{equation}\label{eq:finite-spectral-limit}
 \cU_{N,\rho}(K_M;X^{1:N})
 \xrightarrow[N\to\infty]{\mathrm{law}}\cQ_M.
\end{equation}

We next let \(M\to\infty\).  Since the statistic is linear in its kernel,
\[
 \cU_{N,\rho}(K;X^{1:N})-\cU_{N,\rho}(K_M;X^{1:N})
 =\cU_{N,\rho}(K-K_M;X^{1:N}).
\]
Both \(K\) and \(K_M\) are centered in each variable.  When the square of the
right-hand side is expanded, every cross term therefore vanishes.  Indeed,
terms involving disjoint pairs vanish by independence, while terms involving
pairs with one common particle vanish after conditioning on that particle.
Consequently,
\begin{equation}\label{eq:particle-spectral-tail}
 \begin{aligned}
 &\E\abs{\cU_{N,\rho}(K;X^{1:N})
              -\cU_{N,\rho}(K_M;X^{1:N})}^2\\
 &\qquad=\frac{\binom N2}{(N-1)^2}
   \norm{K-K_M}_{L^2(\rho^{\otimes2})}^2
 =\frac{N}{2(N-1)}\sum_{j>M}\kappa_j^2
 \le\sum_{j>M}\kappa_j^2.
 \end{aligned}
\end{equation}

We now combine the fixed-\(M\) limit with the two remainder estimates.  For
every bounded Lipschitz function \(\Psi:\mathbb R\to\mathbb R\),
\[
 \begin{aligned}
 &\abs{\E\Psi\bigl(\cU_{N,\rho}(K;X^{1:N})\bigr)
       -\E\Psi(\cQ_{g,\rho})}\\
 &\quad\le \operatorname{Lip}(\Psi)
   \E\abs{\cU_{N,\rho}(K;X^{1:N})
            -\cU_{N,\rho}(K_M;X^{1:N})}\\
 &\qquad+\abs{\E\Psi\bigl(\cU_{N,\rho}(K_M;X^{1:N})\bigr)
                  -\E\Psi(\cQ_M)}
   +\operatorname{Lip}(\Psi)\E\abs{\cQ_M-\cQ_{g,\rho}}.
 \end{aligned}
\]
For fixed \(M\), the middle term tends to zero by
\eqref{eq:finite-spectral-limit}.  Applying the Cauchy--Schwarz inequality to
the other two terms and using \eqref{eq:gaussian-chaos-tail} and
\eqref{eq:particle-spectral-tail}, we obtain
\[
 \begin{aligned}
 &\limsup_{N\to\infty}
 \abs{\E\Psi\bigl(\cU_{N,\rho}(K;X^{1:N})\bigr)
       -\E\Psi(\cQ_{g,\rho})}\\
 &\qquad\le \operatorname{Lip}(\Psi)
 \left(1+\frac1{\sqrt2}\right)
 \left(\sum_{j>M}\kappa_j^2\right)^{1/2}
 \xrightarrow[M\to\infty]{}0.
 \end{aligned}
\]
Since bounded Lipschitz test functions determine weak convergence, this proves
\eqref{eq:gaussian-chaos-limit}.

We now prove the convergence of the partition functions.  Fix \(\beta>0\).
Applying Theorem~\ref{thm:uniform-partition} at \(2\beta\), we obtain
\[
 \sup_{N\ge2}\E_{\rho^{\otimes N}}
 e^{-2\beta\cU_{N,\rho}(g;X^{1:N})}
 =\sup_{N\ge2}\cZ_{N,\rho}(2\beta;g)<\infty.
\]
Thus \((e^{-\beta\cU_{N,\rho}(g;X^{1:N})})_{N\ge2}\) is uniformly
integrable.  Moreover, \eqref{eq:gaussian-chaos-limit} and the continuity of
\(x\mapsto e^{-\beta x}\) imply
\[
 e^{-\beta\cU_{N,\rho}(g;X^{1:N})}
 \xrightarrow[N\to\infty]{\mathrm{law}}e^{-\beta\cQ_{g,\rho}}.
\]
Combining this weak convergence with uniform integrability, we obtain
\begin{equation}\label{eq:partition-to-gaussian-limit}
 \lim_{N\to\infty}\cZ_{N,\rho}(\beta;g)
 =\E e^{-\beta\cQ_{g,\rho}}.
\end{equation}

It remains to compute the expectation on the right-hand side.  For a standard
Gaussian variable \(Z\) and \(\kappa\ge0\),
\[
 \E\exp\left(-\frac{\beta\kappa}{2}(Z^2-1)\right)
 =e^{\beta\kappa/2}(1+\beta\kappa)^{-1/2}.
\]
The independence of \((Z_j)_{j\ge1}\) therefore implies
\[
 \E e^{-\beta\cQ_M}
 =\prod_{j=1}^M e^{\beta\kappa_j/2}(1+\beta\kappa_j)^{-1/2}.
\]
Using \(0\le x-\log(1+x)\le x^2/2\) for \(x\ge0\), we also obtain
\[
 \sup_{M\ge1}\E e^{-2\beta\cQ_M}
 \le\exp\left(\beta^2\sum_{j\ge1}\kappa_j^2\right)<\infty.
\]
Thus \((e^{-\beta\cQ_M})_{M\ge1}\) is uniformly integrable.  Since
\(\cQ_M\to\cQ_{g,\rho}\) in \(L^2\), the exponentials converge in probability.
Uniform integrability then allows us to let \(M\to\infty\), and we obtain
\[
 \E e^{-\beta\cQ_{g,\rho}}
 =\prod_{j\ge1}e^{\beta\kappa_j/2}(1+\beta\kappa_j)^{-1/2}
 =\dettwo(\Id+\beta A_{g,\rho})^{-1/2},
\]
where the last identity is the definition of the Carleman--Fredholm determinant
\cite[Chapter~9]{Simon2005}.  The same inequality shows that
\(\sum_{j\ge1}(\beta\kappa_j-\log(1+\beta\kappa_j))\) converges absolutely.
Combining this identity with \eqref{eq:partition-to-gaussian-limit} proves
\eqref{eq:determinant-limit}, and taking logarithms proves
\eqref{eq:determinant-log}.
\end{proof}

\subsection{Sharp dependence on \texorpdfstring{\(\beta\)}{beta}}

We treat the small- and large-\(\beta\) regimes separately.  For small
\(\beta\), the expansion follows by considering each term in
\eqref{eq:determinant-log}.  Indeed, for every \(j\),
\[
 \lim_{\beta\downarrow0}
 \frac{\beta\kappa_j-\log(1+\beta\kappa_j)}{\beta^2}
 =\frac{\kappa_j^2}{2},
 \qquad
 0\le
 \frac{\beta\kappa_j-\log(1+\beta\kappa_j)}{\beta^2}
 \le\frac{\kappa_j^2}{2}.
\]
The sequence \((\kappa_j^2/2)_{j\ge1}\) is summable by
\eqref{eq:kernel-spectral-decomposition}.  Dividing
\eqref{eq:determinant-log} by \(\beta^2\) and applying the dominated
convergence theorem to the series, we obtain
\[
 \lim_{\beta\downarrow0}
 \frac{\log\cZ_{\infty,\rho}(\beta;g)}{\beta^2}
 =\frac14\sum_{j\ge1}\kappa_j^2
 =\frac14\norm{g_\rho^\circ}_{L^2(\rho^{\otimes2})}^2.
\]
Equivalently,
\begin{equation}\label{eq:sharp-small-beta}
 \log\cZ_{\infty,\rho}(\beta;g)
 =\frac{\beta^2}{4}
  \norm{g_\rho^\circ}_{L^2(\rho^{\otimes2})}^2+o(\beta^2),
 \qquad \beta\downarrow0.
\end{equation}
If \(g_\rho^\circ\ne0\), the coefficient in
\eqref{eq:sharp-small-beta} is positive.  Combining
\eqref{eq:sharp-small-beta}, \eqref{eq:determinant-limit}, and the
small-\(\beta\) bound in Theorem~\ref{thm:uniform-partition}, we obtain, for
all sufficiently small \(\beta>0\),
\[
 c_{g,\rho}\beta^2
 \le\log\cZ_{\infty,\rho}(\beta;g)
 \le\sup_{N\ge2}\log\cZ_{N,\rho}(\beta;g)
 \le C_{d,s,m,\rho}\beta^2.
\]
Thus the quadratic dependence on \(\beta\) in
Theorem~\ref{thm:uniform-partition} is sharp whenever
\(g_\rho^\circ\ne0\).

We now turn to large \(\beta\).  The upper bounds follow from
\eqref{eq:riesz-parameter-bound}.  By \eqref{eq:determinant-log}, the matching
lower bounds reduce to estimating the same eigenvalue sequence
\((\kappa_j)_{j\ge1}\) as in \eqref{eq:kernel-spectral-decomposition}.

\begin{lemma}
\label{lem:model-eigen-lower}
Let \(g=g_{d,s}^{(m)}\), where \(m=0\) when \(s=0\), and let
\((\kappa_j)_{j\ge1}\) be the nonincreasing eigenvalue sequence used in
\eqref{eq:kernel-spectral-decomposition}.
Assume that \(\rho\) is bounded below by a positive constant almost everywhere
on a ball \(D\subset\mathbb R^d\).  Then
\begin{equation*}
 \kappa_j\ge c_{d,s,m,\rho,D}j^{-(d-s)/d},
 \qquad j\ge1.
\end{equation*}
In particular, \(A_{g,\rho}\) has infinitely many positive eigenvalues.
\end{lemma}

\begin{proof}
By Lemma~\ref{lem:kernel-operator-properties}, \(A_{g,\rho}\) is a
nonnegative compact self-adjoint operator, so \(\kappa_j\ge0\) for every
\(j\).  The max--min principle for compact self-adjoint operators
\cite[Theorem~4.22]{Borthwick2020} gives, for every
\(n\)-dimensional subspace \(E\subset L^2(\rho)\),
\begin{equation}\label{eq:max-min-consequence}
 \kappa_n\ge
 \inf_{0\ne f\in E}
 \frac{\langle f,A_{g,\rho}f\rangle_{L^2(\rho)}}
      {\|f\|_{L^2(\rho)}^2}.
\end{equation}
We now construct such subspaces using functions supported in \(D\).

Choose a cube \(Q=x_0+(-\ell,\ell)^d\Subset D\) and fix once and for all a
nonzero real function \(\psi\in C_c^\infty((-1/2,1/2)^d)\) such that
\(\int\psi=0\).  For an integer \(L\ge1\), divide \(Q\) into \(L^d\) equal
cubes.  Let \(x_{L,q}\) be their
centers and let \(h_L=2\ell/L\) be their common side length.  Define
\[
 \psi_{L,q}(x):=h_L^{-d/2}
 \psi\left(\frac{x-x_{L,q}}{h_L}\right),
 \qquad 1\le q\le L^d.
\]
Let
\[
 V_L:=\left\{
 u=\sum_{q=1}^{L^d}a_q\psi_{L,q}:
 a_1,\ldots,a_{L^d}\in\mathbb R
 \right\}.
\]
Thus, \(V_L\) consists of all real linear combinations of these functions.
Since their supports are disjoint,
\[
 \left\|\sum_{q=1}^{L^d}a_q\psi_{L,q}\right\|_2^2
 =\sum_{q=1}^{L^d}a_q^2\|\psi_{L,q}\|_2^2.
\]
Each \(\psi_{L,q}\) is nonzero, so the linear combination on the left can
vanish only when all the coefficients are zero.  Thus every element of
\(V_L\) has a unique representation with \(L^d\) coefficients, and hence
\(\dim V_L=L^d\).  Finally, each \(\psi_{L,q}\) has zero integral, so every
\(u\in V_L\) has zero mean.

Fix an integer \(r>(d-s)/2\).  For every multi-index \(\gamma\) with
\(|\gamma|\le r\), the functions \(\partial^\gamma\psi_{L,q}\) still have
disjoint supports and are therefore orthogonal in \(L^2\).  Thus, if
\(u=\sum_{q=1}^{L^d}a_q\psi_{L,q}\), the cross terms vanish, and a change of
variables gives
\[
 \|\partial^\gamma u\|_2^2
 =\sum_{q=1}^{L^d}a_q^2\|\partial^\gamma\psi_{L,q}\|_2^2
 =h_L^{-2|\gamma|}
 \frac{\|\partial^\gamma\psi\|_2^2}{\|\psi\|_2^2}\|u\|_2^2.
\]
Summing over \(|\gamma|\le r\) and using \(h_L=2\ell/L\), we obtain
\[
 \|u\|_{H^r}^2
 \le C_{d,r,\psi}(1+h_L^{-2r})\|u\|_2^2
 \le C_{d,r,D}(1+L)^{2r}\|u\|_2^2,
 \qquad u\in V_L.
\]
Here the fixed template \(\psi\) and the fixed value of \(\ell\) are absorbed
into \(C_{d,r,D}\), which is independent of \(L\).  With
\(\theta=(d-s)/(2r)\), interpolation between \(L^2\) and \(H^r\) then gives
\[
 \|u\|_{H^{(d-s)/2}}^2
 \le C_{d,r}\|u\|_2^{2(1-\theta)}\|u\|_{H^r}^{2\theta}
 \le C_{d,s,D}(1+L)^{d-s}\|u\|_2^2.
\]
Since
\[
 (m^2+|\xi|^2)^{(d-s)/2}
 \le C_{d,s,m}(1+|\xi|^2)^{(d-s)/2},
\]
the Fourier characterization of \(H^{(d-s)/2}\) yields
\begin{equation}\label{eq:localized-sobolev-cost}
 (2\pi)^{-d}\int_{\mathbb R^d}
 (m^2+|\xi|^2)^{(d-s)/2}|\widehat u(\xi)|^2\dd\xi
 \le C_{d,s,m,D}(1+L)^{d-s}\|u\|_2^2.
\end{equation}

For \(0\ne u\in V_L\), the Fourier representation of \(g\) gives
\[
 \iint g(x-y)u(x)u(y)\dd x\dd y
 =(2\pi)^{-d}\int_{\mathbb R^d}
 (m^2+|\xi|^2)^{-(d-s)/2}|\widehat u(\xi)|^2\dd\xi.
\]
By the Cauchy--Schwarz inequality and
\eqref{eq:localized-sobolev-cost},
\[
\begin{aligned}
 \|u\|_2^4
 &\le
 \left((2\pi)^{-d}\int_{\mathbb R^d}
 (m^2+|\xi|^2)^{-(d-s)/2}|\widehat u(\xi)|^2\dd\xi\right)\\
 &\qquad\times
 \left((2\pi)^{-d}\int_{\mathbb R^d}
 (m^2+|\xi|^2)^{(d-s)/2}|\widehat u(\xi)|^2\dd\xi\right).
\end{aligned}
\]
The first factor on the right-hand side is the interaction energy, while the
second is bounded by \eqref{eq:localized-sobolev-cost}.  Therefore,
\begin{equation}\label{eq:localized-interaction-lower}
 \iint g(x-y)u(x)u(y)\dd x\dd y
 \ge c_{d,s,m,D}(1+L)^{-(d-s)}\|u\|_2^2.
\end{equation}

We now use these functions to test \(A_{g,\rho}\).  For \(u\in V_L\), define
\(f=u/\rho\) on \(D\) and \(f=0\) outside \(D\).  By the positive lower
bound on \(\rho\) in \(D\),
\begin{equation}\label{eq:localized-f-bound}
 \|f\|_{L^2(\rho)}^2
 =\int_D\frac{|u|^2}{\rho}
 \le C_{\rho,D}\|u\|_2^2,
 \qquad
 \int_{\mathbb R^d}f\rho=0.
\end{equation}
By the definition of \(A_{g,\rho}\) and the identity \(f\rho=u\),
\[
\begin{aligned}
 \langle f,A_{g,\rho}f\rangle_{L^2(\rho)}
 &=\iint g_\rho^\circ(x,y)u(x)u(y)\dd x\dd y\\
 &=\iint g(x-y)u(x)u(y)\dd x\dd y.
\end{aligned}
\]
Here the second equality follows from \(\int u=0\) and the definition of
\(g_\rho^\circ\).  Combining \eqref{eq:localized-interaction-lower} and
\eqref{eq:localized-f-bound}, we obtain
\[
 \langle f,A_{g,\rho}f\rangle_{L^2(\rho)}
 \ge c_{d,s,m,\rho,D}(1+L)^{-(d-s)}
 \|f\|_{L^2(\rho)}^2.
\]
The map \(u\mapsto f\) is injective, so its image has dimension \(L^d\).
Applying \eqref{eq:max-min-consequence} to this \(L^d\)-dimensional subspace,
we obtain
\[
 \kappa_{L^d}\ge c_{d,s,m,\rho,D}(1+L)^{-(d-s)}.
\]
Given \(n\ge1\), take \(L=\lceil n^{1/d}\rceil\).  Then \(n\le L^d\) and
\(1+L\le3n^{1/d}\).  Since \((\kappa_j)_{j\ge1}\) is nonincreasing,
\[
 \kappa_n\ge\kappa_{L^d}
 \ge c_{d,s,m,\rho,D}(1+L)^{-(d-s)}
 \ge c_{d,s,m,\rho,D}n^{-(d-s)/d}.
\]
Thus every \(\kappa_n\) is positive, which also proves that \(A_{g,\rho}\)
has infinitely many positive eigenvalues.
\end{proof}

\begin{proof}[Proof of Theorem~\ref{thm:sharp-coulomb-growth}]
Let \(g=g_{d,s}^{(m)}\), with \(m=0\) when \(s=0\).  For all sufficiently
large \(\beta\), \eqref{eq:determinant-log} and
\eqref{eq:riesz-parameter-bound} give
\begin{equation}\label{eq:large-beta-upper}
 \log\cZ_{\infty,\rho}(\beta;g)
 \le\sup_{N\ge2}\log\cZ_{N,\rho}(\beta;g)
 \le
 \begin{cases}
  C_{d,\rho}\beta\log\beta,&s=0,\\
  C_{d,s,m,\rho}\beta^{d/(d-s)},&0<s<d/2.
 \end{cases}
\end{equation}
In the logarithmic case, we also used
\(\log(2+\beta)\le C\log\beta\).  It remains to prove the lower bounds.

Choose a sufficiently small constant \(c>0\), depending only on
\(d,s,m,\rho,D\), and set
\[
 M_\beta:=\left\lfloor c\beta^{d/(d-s)}\right\rfloor.
\]
For \(1\le j\le M_\beta\), Lemma~\ref{lem:model-eigen-lower} gives
\[
 \beta\kappa_j
 \ge c_{d,s,m,\rho,D}\beta M_\beta^{-(d-s)/d}.
\]
By the definition of \(M_\beta\), the right-hand side is at least \(1\)
if \(c\) is sufficiently small.  Therefore, for all sufficiently large
\(\beta\), we have \(M_\beta\ge1\) and
\begin{equation}\label{eq:large-beta-active-modes}
 \beta\kappa_j\ge1,
 \qquad 1\le j\le M_\beta.
\end{equation}
Since \(x-\log(1+x)\ge0\) for \(x\ge0\) and
\(x-\log(1+x)\ge x/4\) for \(x\ge1\), applying
\eqref{eq:large-beta-active-modes} to the first \(M_\beta\) terms in
\eqref{eq:determinant-log} and then using
Lemma~\ref{lem:model-eigen-lower}, we obtain
\begin{equation}\label{eq:large-beta-common-lower}
 \log\cZ_{\infty,\rho}(\beta;g)
 \ge\frac{\beta}{8}\sum_{j=1}^{M_\beta}\kappa_j
 \ge c_{d,s,m,\rho,D}\beta
       \sum_{j=1}^{M_\beta}j^{-(d-s)/d}.
\end{equation}

If \(s=0\), then \(M_\beta=\lfloor c\beta\rfloor\), and
\[
 \sum_{j=1}^{M_\beta}\frac1j
 \ge\log(M_\beta+1)
 \ge\frac12\log\beta.
\]
Thus \eqref{eq:large-beta-common-lower} gives
\[
 \log\cZ_{\infty,\rho}(\beta;g_{d,0})
 \ge c_{d,\rho,D}\beta\log\beta.
\]
If \(0<s<d/2\), then
\[
 \sum_{j=1}^{M_\beta}j^{-(d-s)/d}
 \ge c_{d,s}M_\beta^{s/d}.
\]
Therefore, by \eqref{eq:large-beta-common-lower},
\[
 \log\cZ_{\infty,\rho}(\beta;g)
 \ge c_{d,s,m,\rho,D}\beta M_\beta^{s/d}
 \ge c_{d,s,m,\rho,D}\beta^{d/(d-s)}.
\]
Combining these estimates with \eqref{eq:large-beta-upper}, we obtain
\eqref{eq:sharp-log-growth} and \eqref{eq:sharp-riesz-growth}.
\end{proof}

\bibliographystyle{abbrv}
\bibliography{WZ_Coulomb_V10}

\end{document}